\documentclass[final, a4paper, 11pt]{article}

\usepackage[scale=0.768]{geometry}
\usepackage[english]{babel}

\usepackage[mathscr]{euscript}
\usepackage{color}
\usepackage{fancyhdr}
\usepackage{bm}
\usepackage{hyperref}
\usepackage{cite}
\usepackage{showkeys}
\usepackage{subfigure}
\usepackage{float}

\usepackage{amsfonts}
\usepackage{amsmath}
\usepackage{amssymb}
\usepackage{amscd}
\usepackage{amsthm}

\usepackage{accents}
\usepackage{graphicx}
\usepackage{array}

\newtheorem{theorem}{Theorem}
\newtheorem{definition}[theorem]{Definition}
\newtheorem{lemma}[theorem]{Lemma}

\newtheorem{remark}[theorem]{Remark}
\newtheorem{assumption}[theorem]{Assumption}

\newcommand*{\N}{\ensuremath{\mathbb{N}}}
\newcommand*{\Z}{\ensuremath{\mathbb{Z}}}
\newcommand*{\R}{\ensuremath{\mathbb{R}}}
\newcommand*{\C}{\ensuremath{\mathbb{C}}}
\renewcommand{\i}{\mathrm{i}}
\renewcommand{\phi}{\varphi}
\renewcommand{\rho}{{\varrho}}
\renewcommand{\epsilon}{{\varepsilon}}
\DeclareMathOperator*{\esssup}{ess\,sup}
\renewcommand{\d}[1]{\,\mathrm{d}#1 \,}

\newcommand{\D}{\mathcal{D}}

\newcommand{\J}{\mathcal{J}} % Bloch transform
\newcommand{\0}{{0}} %  {0}
\newcommand{\grad}{\nabla}

\newcommand{\W}{{W_{\hspace*{-1pt}{\Lambda}}}} % Wigner-Seitz cell
\newcommand{\Wast}{{W_{\hspace*{-1pt}{\Lambda}^\ast}}} % Brillouin zone

\newlength{\dhatheight}

\usepackage{color}
\newcommand{\high}[1]{{\color{black}{#1}}}

\begin{document}

\sloppy

\title{\high{Numerical methods for scattering} problems from multi-layers with different periodicities}
\author{Ruming Zhang\thanks{Institute for Applied and Numerical Mathematics, Karlsruhe Institute of Technology, Karlsruhe, Germany
; \texttt{ruming.zhang@kit.edu}}}
\date{}
\maketitle

\begin{abstract}
In this paper, we consider a numerical method to solve scattering problems with multi-periodic layers with different periodicities. The main tool applied in this paper is the Bloch transform. With this method, the problem is written into an equivalent coupled family of quasi-periodic problems. As the Bloch transform is only defined for one fixed period, the inhomogeneous layer with another period is simply treated as a non-periodic one. First, we approximate the refractive index by a periodic one where its period is an integer multiple of the fixed period, and it is decomposed by finite number of quasi-periodic functions. Then the coupled system is reduced into a simplified formulation. 
A convergent finite element method is proposed for the numerical solution, and the numerical method has been applied to several numerical experiments. At the end of this paper, relative errors of the numerical solutions will be shown to illustrate the convergence of the numerical algorithm.
\end{abstract}

\section{Introduction}

In this paper, we develop a numerical method to solve acoustic scattering problems with two-layer structures in 2D spaces, where each layer is periodic with different periodicities. This is a simplified \high{model} of the design of microstrip array antennas in 3D ( see \cite{Bhatt2000}). The easier case, for example, when either the periodicities are the same, or the quotient of the periodicities is rational, the problem is naturally reduced into a problem with one periodic layer, which is easily treated in the classic frame work for quasi-periodic scattering problems (see \cite{Stryc1998,Lechl2016}). However, if the quotient of the periodicities is either irrational or extremely large/small, the problem becomes much more complicated. For the first case, the original problem is impossible to be reduced into any bounded domain naturally, thus it is a scattering problem with unbounded inhomogeneous medium; while for the second case, although the problem could be reduced into one periodic cell, the cell will be very large. For both cases,  numerical simulations of these problems are very challenging.

Scattering problems with unbounded structures has been investigated by many mathematicians in  decades. Based on the integral equation method, the well-posedness of these scattering problems has been established (see \cite{Chand1996,Chand1999,Chand1998,Zhang2003}), and numerical methods have been proposed for rough surface scattering problems (see \cite{Meier2000,Chand2002,Arens2006a}). The variational method, on the other hand, has also been applied to theoretical analysis of scattering from unbounded obstacles (see \cite{Chand2005,Chand2007,Lechl2009,Li2012}). An important extension of the variational method is to consider the well-posedness in weighted Sobolev spaces (see \cite{Chand2010}), and more generalized cases (e.g. incident plane waves) are included. Similar results in weighted Sobolev spaces \high{have} been shown for more generalized boundary conditions in \cite{Hu2015}.

Recently, a Floquet-Bloch transform based method has been proposed for the study of scattering problems with unbounded structures, especially for structures that are either periodic or slightly different from  periodic ones. As far as the author knows, the first paper that adopted this method is \cite{Coatl2012} for scattering problems with locally perturbed periodic mediums. Inspired by this paper, the method has been extended to scattering problems with non-periodic incident fields with (locally perturbed) periodic  surfaces (see \cite{Lechl2015e,Lechl2017,Hadda2015}). Based on the theoretical results, Bloch-transform based numerical methods have been proposed (see  \cite{Lechl2016a,Lechl2016b,Lechl2017}.  The Bloch transform was also applied to other cases, i.e., scattering problems in locally perturbed periodic waveguides, see \cite{Fliss2015}. For all these works listed above, the perturbations of periodic surfaces or inhomogeneous mediums are assumed to be compactly supported. In this case, the Bloch transformed problem has a simplified variational form. However, for more general cases, i.e., when the perturbations are non-compactly supported, the problems become much more complicated and difficult to be dealt with. Further study on the Bloch transform is then required for the globally perturbed problems.

In this paper, the Bloch transformed scattering problems from different periodic layers in $\R^2$ will be investigated. The original problem is approximated by a new one with a periodic layer, and the weak formulation for the Bloch transformed new problem is established, and the equivalence, well-posedness and regularity results are proved following \cite{Lechl2016}. Based on the weak formulation, the numerical method will be introduced. The key step is the approximation of periodic inohomogeneous media by a finite series of quasi-periodic functions with another different period.  The inhomogeneous media is first approximated by a periodic one with a relatively larger period, and the compactly supported function is then approximated by a finite Fourier series. With the method inspired by the decomposition (52) in \cite{Hadda2017}, the Fourier series is written into the sum of finite number of quasi-periodic functions.

The rest of the paper is organized as follows. In Section 2, we will describe the mathematical model of the scattering problems and show the well-posedness of the problem. In Section 3, we approximate the original scattering problem by replacing the inhomogeneous layer with a periodic one. Then we apply the Bloch transform to the new problem in Section 4. In Section 5 and 6, we formulate the discretization of the transformed problem. Finally, we show some numerical examples in the last section.

\section{Scattering problems: mathematical model}

In this section, we describe the mathematical modal for scattering problems with periodic layers with different periods in two dimensional spaces (see Figure \ref{sample}). Let the straight line $\Gamma_h:=\R\times\{h\}$ for any $h\in\R$, and assume that $\Gamma_{h_0}$ where $h_0>0$ is a sound-soft surface. Define the domains  by
\begin{equation*}
D:=\R\times(h_0,\infty);\, \quad D_H=\R\times(h_0,H)
\end{equation*}
where $H>h_0$. Assume that the infinite layer is embedded in $D_H$ for some fixed positive number $H$, and it is divided into two layers by a straight line $\Gamma_{H_1}$, for some $H_1\in(h_0,H)$. Let $D_1=\R\times(h_0,H_1)$ and $D_2=\R\times(H_1,H)$. Let
\begin{equation*}
n(x_1,x_2)=\begin{cases}
n_1(x_1,x_2),\quad\text{ when } x\in D_1;\\
n_2(x_1,x_2),\quad\text{ when } x\in D_2;\\
0,\quad\text{ when }x_2\geq H,
\end{cases}
\end{equation*}
where $n_1$ and $n_2$ is are both periodic functions in $x_1$-direction. The period of $n_1$ is $\Lambda_1>0$ and that of $n_2$ is $\Lambda>0$. We simply assume that $\Lambda_1\neq \Lambda$ without further conditions.

\begin{remark}
  $n$ is simply assumed to be in the space $L^{\infty}(D)$. However, to guarantee the convergence of the numerical method, we may assume that the refractive index has a higher regularity later.
\end{remark}

\begin{figure}[H]
\centering
\includegraphics[width=12cm]{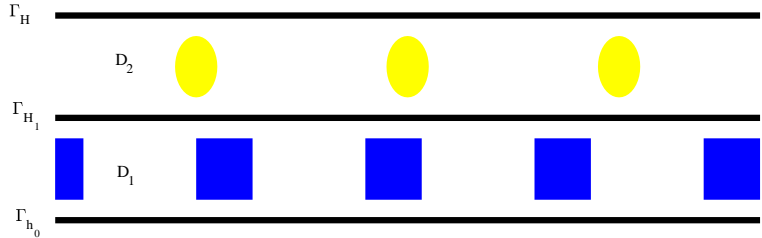}
\caption{Inhomogeneous layers with different periodicities.}
\label{sample}
\end{figure}

Consider a scattering problem with an inhomogeneous medium, which is modelled by the Helmholtz equation with a homogeneous Dirichlet boundary condition on $\Gamma_{h_0}$\high{:}
\begin{equation}\label{eq:helmholtz}
\Delta u+k^2 (1+n) u=g\quad\text{ in }D,\quad u=0\text{ on }\Gamma_{h_0},
\end{equation}
where $g$ is the source term supported in $D_H$. To guarantee that the solution $u$ is upward propagating, it is required that $u$ satisfies the following boundary condition on $\Gamma_H$
\begin{equation}\label{eq:boundary}
\frac{\partial u}{\partial x_2}(x_1,H)=T^+\left[u\big|_{\Gamma_H}\right],
\end{equation}
where $T^+$ is the Dirichlet-to-Neumann map that maps $H^{1/2}(\Gamma_H)$ to $H^{-1/2}(\Gamma_H)$  (see \cite{Chand2005}), and it is defined by
\begin{equation}\label{eq:DtN}
T^+\phi=\frac{\i}{{2\pi}}\int_{\R} \sqrt{k^2-|{\xi}|^s}e^{\i x_1\cdot{\xi}}\,\widehat{\phi}({\xi})\d{\xi},\quad\text{ where }\phi=\frac{1}{2\pi}\int_{\R} e^{\i x_1\cdot{\xi}}\,\widehat{\phi}({\xi})\d{\xi}.
\end{equation}
The scattering problem is now formulated into the one that is defined \high{on}  the domain $\high{D}_H$ with finite height.  The weak formulation for the scattering problem is, given any $g\in H^{-1}(D_H)$, to find a solution $u\in\widetilde{H}^1(\high{D}_H)$ such that
\begin{equation}\label{eq:var_origional}
\int_{D_H}\left[\nabla u\cdot\nabla\overline{v}-k^2(1+n) u\overline{v}\right]\d x-\int_{\Gamma_H}T^+\left[u\big|_{\Gamma_H}\right]\overline{v}\d s=-\int_{D_H}g\overline{v}\d x,
\end{equation}
for all $v\in\widetilde{H}^1(D_H)$ with compact support in $\overline{D_H}$.  Note that the tilde in $\widetilde{H}^1(D_H)$ shows that the functions in this space belong to $H^1(D_H)$ and satisfy the homogeneous Dirichlet boundary condition on $\Gamma_{h_0}$.  Similar notations are adopted for other spaces, e.g., $\widetilde{H}_r^1(D_H)$ and $H_0^r(\Wast;\widetilde{H}_{\alpha}^s(D^\Lambda_H))$, in the following parts of this paper.

\high{
Following \cite{Chand2010}, we consider the solution of the scattering problem in weighted Sobolev spaces. Define the weighted Sobolev space in $D_H$ for any fixed $r\in\R$ by:
\begin{equation*}
 H_r^s(D_H):=\left\{\phi\in\mathcal{D}'(D_H):\,(1+|x|^2)^{r/2}\phi(x)\in H^s(D_H)\right\}.
\end{equation*}
The definitions for $H^{1/2}_r(\Gamma_H)$ and $H^{-1/2}_r(\Gamma_H)$ are similar. 
}

From \cite{Chand2010} again, the operator $T^+$ is bounded from $H^{1/2}_r(\Gamma_H)$ to $H^{-1/2}_r(\Gamma_H)$ for any $|r|<1$, thus the left-hand-side of \eqref{eq:var_origional} is a bounded sesquilinear form defined in $\widetilde{H}_r^1(D_H)\times {H}_{-r}^1(D_H)$. For any $g\in H^{-1}_r(D_H)$, we are looking for a solution $u\in \widetilde{H}_r^1(D_H)$ such that \eqref{eq:var_origional} holds for any $v\in \widetilde{H}_{-r}^1(D_H)$. From Riesz's lemma, there is a bounded linear operator depending on $n$, i.e., $\mathcal{B}_r(n):\, {H}_r^1(D_H)\rightarrow({H}_{\high{-}r}^1(D_H))^*$, such that
\begin{equation*}
\int_{D_H}\left[\nabla u\cdot\nabla\overline{v}-k^2(1+n) u\overline{v}\right]\d x-\int_{\Gamma_H}T^+\left[u\big|_{\Gamma_H}\right]\overline{v}\d s=\left(\mathcal{B}_r(n) u,v\right)_{(H_{-r}^1(D_H))^*\times H_{-r}^1(D_H)}.
\end{equation*}
Especially, when $n=0$ in $D$, the problem is reduced to the scattering problem from the sound soft surface $\Gamma_{h_0}$ with homogeneous media in $D$. The well-posedness for this problem in the space $\widetilde{H}_r^1(D_H)$ has been proved in \cite{Chand2010}, thus the operator $\mathcal{B}_r(0)$ is invertible. Then the operator 
\begin{equation*}
\mathcal{B}_r(n):=\mathcal{B}_r(0)+\big[\mathcal{B}_r(n)-\mathcal{B}_r(0)\big]
\end{equation*}
is a perturbation of the isomorphism  $\mathcal{B}_r(0)$. The perturbation $\mathcal{B}_r(n)-\mathcal{B}_r(0)$ satisfies
\begin{equation*}
\left(\left[\mathcal{B}_r(n)-\mathcal{B}_r(0)\right]u,v\right)_{H_r^1(D_H)\times H_{-r}^1(D_H)}=-k^2\int_{D_H}n u\overline{v}\d x.
\end{equation*}

\begin{lemma}
 The operator $\mathcal{K}_r(n):=\mathcal{B}_r(n)-\mathcal{B}_r(0)$ is bounded from $H_r^1(D_H)$ to $\left(H_{-r}^1(D_H)\right)^*$, and the norm is bounded by
 \begin{equation*}
  \|\mathcal{K}_r(n)\|\leq k^2\|n\|_{\infty},
 \end{equation*}
where $\|\cdot\|$ is the operator norm and $\|\cdot\|_\infty$ is the $L^\infty(D_H)$ norm.
\end{lemma}

The proof is trivial thus is omitted.

\high{
As $\mathcal{B}_r(0)$ is invertible and $\mathcal{K}_r(n)$ is bounded by $k^2\|n\|_{\infty}$,  when $k^2\|n\|_{\infty}$ is small enough, $\mathcal{B}_r(n)=\mathcal{B}_r(0)+\mathcal{K}_r(n)$ is invertible. We conclude the well-posedness result for \eqref{eq:var_origional} in the following theorem.

\begin{theorem}\label{th:solv}
Suppose  $k^2\|n\|_{\infty}$ is small enough, i.e., $k^2\|n\|_{\infty}\leq \left\|\mathcal{B}_r(0)^{-1}\right\|^{-1}$.
Given any function $g\in H_r^{-1}(D_H)$ for some fixed $|r|<1$, the variational problem \eqref{eq:var_origional} is uniquely solvable in the space $\widetilde{H}_r^1(D_H)$. Moreover, there is a constant that depends on $k$ and $n$ such that
\begin{equation}\label{eq:stab}
 \|u\|_{H_r^1(D_H)}\leq C\|g\|_{H^{-1}_r(D_H)}.
\end{equation}

\end{theorem}

\begin{remark}
The condition in Theorem \ref{th:solv} is not optimal. In fact, a number of research papers are devoted to the well-posedness of the scattering problems from rough layers, for details we refer to \cite{Zhang1998,Chand1999a,Chand2007,Lechl2009} for Helmholtz equations and  \cite{Chand1998a,Hadda2011} for Maxwell's equations.  However, in this paper, as we are only interested in the numerical solutions for this kind of problems, we simply assume that $k^2\|n\|_\infty$ is small enough to  guarantee that the problem \eqref{eq:var_origional} is uniquely solvable, and the unique solution satisfies \eqref{eq:stab}.
\end{remark}
}

\section{Approximation of solutions with unbounded refractive index}

To solve a problem defined in an unbounded domain, it is natural to approximate it by one defined in a bounded one. However, for this case, as the refractive index has two different periodic layers, we would like to approximate it by a  periodic one. Thus we fix one periodic layer, and modify another layer based on the period of the fixed layer. 

Let $N>0$ be a sufficiently large integer, and the smooth cutoff function $\mathcal{X}(t)$ satisfies
\begin{equation*}
 \mathcal{X}(t)=\begin{cases}
                 1,\quad |t|\leq N\Lambda/4;\\
                 0,\quad |t|\geq N\Lambda/2;\\
                 \text{smooth},\quad\text{othewise}.
                \end{cases}
\end{equation*}
We define a new function by
\begin{equation*}
 n_1^N(x_1,x_2):=n_1(x_1,x_2)\mathcal{X}(x_1),\quad -\frac{N\Lambda}{2}\leq x_1\leq \frac{N\Lambda}{2}.
\end{equation*}
We extend $n_1^N$ into an $N\Lambda$-periodic function in $x_1$-direction, and it is still denoted by $n_1^N$. Let 
\begin{equation*}
 n_N(x)=\begin{cases}
                   n_1^N(x),\quad\text{ when }x\in D_1;\\
                   n_2(x),\quad\text{ when }x\in D_2;\\
                   0,\quad\text{ when }x_2\geq H.
                  \end{cases}
\end{equation*}
As $n_2$ is $\Lambda$-periodic and $n_1^N$ is $N\Lambda$-periodic, the function $n_N$ is $N\Lambda$-periodic as well. 
Define $D_H^N:=\{x\in D_1:\,|x_1|\leq N\Lambda/4\}\cup D_2$, then $n=n_N$ when $x\in D_H^N$. When $x\in D_H\setminus D_H^N$, 
\begin{equation*}
 \|n_N\|_{\infty}\leq \|n\|_{\infty}; \quad \|n-n_N\|_{\infty}\leq 2\|n\|_{\infty}.
\end{equation*}

We consider the new variational problem, with $n$ replaced  by $n_N$ in \eqref{eq:var_origional}. Give any $g\in H^{-1}_r(D_H)$, we are looking for a solution $u_N\in\widetilde{H}^1_r(D_H)$ such that
\begin{equation}\label{eq:var_new}
 \int_{D_H}\left[\nabla u_N\cdot\nabla\overline{v}-k^2(1+n_N)u_N\overline{v}\right]\d x-\int_{\Gamma_H}T^+\left[u_N\big|_{\Gamma_H}\right]\overline{v}\d s=-\int_{D_H}g\overline{v}\d x
\end{equation}
holds for any $v\in \widetilde{H}^1_r(D_H)$. From the definition of $\mathcal{B}_r(n)$, the left hand side is equivalent to $\left(\mathcal{B}_r(n_N)u_N,v\right)_{(H_{-r}^1(D_H))^*\times H^1_{-r}(D_H)}$. From the fact that $\|n_N\|_{\infty}\leq \|n\|_{\infty}$, we obtain the invertibility of $\mathcal{B}_r(n_N)$  in the following theorem.

\begin{theorem}\label{th:solv_new}
Suppose $k^2\|n\|_\infty\leq\left\|\mathcal{B}_r(0)^{-1}\right\|^{-1}$. 
 For any $g\in H^{-1}_r(D_H)$, there is a unique solution $u_N\in \widetilde{H}_r^1(D_H)$ such that \eqref{eq:var_new} is satisfied. Moreover,
 \begin{equation}\label{eq:stab_new}
  \|u_N\|_{\widetilde{H}_r^1(D_H)}\leq C\|g\|_{H^{-1}_r(D_H)}
 \end{equation}
 holds uniformly for $N\in\N$,  where $C$ is the same as that in \eqref{eq:stab}.
\end{theorem}

With the result in Theorem \ref{th:solv} and \ref{th:solv_new}, we have the following estimation between $u$ and $u_N$.

\begin{theorem}\label{th:err_approx}
 Suppose $k^2\|n\|_\infty\leq\left\|\mathcal{B}_r(0)^{-1}\right\|^{-1}$. When $N$ is large enough, the error between $u_N$ and $u$ is bounded by
 \begin{equation*}
  \|u_N-u\|_{H^1(D_H)}\leq C(N\Lambda/4)^{-r}\|u\|_{H_r^0(D_H)}
 \end{equation*}
where $C$ is independent of $N$ and $u$.
\end{theorem}

\begin{proof}
 Let $\delta u_N:=u_N-u$, then from \eqref{eq:var_origional} and \eqref{eq:var_new}, it satisfies 
 \begin{equation*}
  \left(\mathcal{B}_r(n_N) \delta u_N,v\right)=k^2\int_{D_H}(n-n_N)u\overline{v}\d x.
 \end{equation*}
 The right hand side is bounded by
 \begin{equation*}
 \begin{aligned}
  \left|k^2\int_{D_H}(n-n_N)u\overline{v}\d x\right|&\leq 2k^2\|n\|_\infty\|u\|_{L^2(D_H\setminus D_H^N)}\|v\|_{L^2(D_H)}\\
  &\leq 2k^2\|n\|_\infty (N\Lambda/4)^{-r}\|u\|_{H_r^0(D_H)}\|v\|_{H^1(D_H)},
  \end{aligned}
 \end{equation*}
 thus it defines a bounded anti-linear functional on $v$. As $\mathcal{B}_r(n_N)$ is invertible and the inverse operator is uniformly bounded with large enough $N$'s,
 \begin{equation*}
  \|\delta u_N\|_{H^1(D_H)}\leq C(N\Lambda/4)^{-r}\|u\|_{H_r^0(D_H)}.
 \end{equation*}

 The proof is finished.

\end{proof}

Now we have approximated the original problem \eqref{eq:var_origional} by the new one with a $N\Lambda$-periodic refractive index $n_N$. We proved that when $r>0$, the $H^1$-norm converges at the rate of $(N\Lambda/4)^{-r}$ when $r\in(0,1)$, as $N\rightarrow\infty$. In the following, we apply the Floquet-Bloch transform to the newly established problem.

\section{The Bloch transform of the approximated problem}

In this section, we apply the Bloch transform (for its definition see Appendix) to analyse the approximated problem \eqref{eq:var_new}. The periodic cell for $x_1$, also called the Wigner-Seitz-cell, is defined by
\begin{equation*}
\W:=\left(-\frac{\Lambda}{2},\frac{\Lambda}{2}\right].
\end{equation*}
Let $\Lambda^*:=2\pi/\Lambda$, then the dual cell of $\W$, i.e., the so called  Brillouin zone, is defined by
\begin{equation*}
\Wast:=\left(0,\Lambda^*\right]=\left(0,\frac{2\pi}{\Lambda}\right].
\end{equation*}
Let $\Gamma^\Lambda_H$ and $D^\Lambda_H$ be restrictions of $\Gamma_H$ and $D_H$ in one periodic cell $\W\times\R$, i.e.,
\begin{equation*}
\Gamma^\Lambda_H=\Gamma_H\cap\left[\W\times\R\right],\quad D^\Lambda_H=D_H\cap\left[\W\times\R\right].
\end{equation*} 
The definitions are similar for other domains restricted in one periodic cell $\W\times\R$. 

Let $n_N$ be the $\Lambda$-periodic function defined by
\begin{equation*}
\widetilde{n}_N=\begin{cases}
1+n_2,\quad x\in D_2;\\
1,\quad \text{otherwise.}
\end{cases}
\end{equation*}
Extend $n_1^N$ by $0$ to the half space $D$, it is still $N\Lambda$-periodic  in $x_1$-direction, then
\begin{equation*}
 1+n_N=\widetilde{n}_N+n_1^N.
\end{equation*}
Use the property of the Bloch transform, let $w=\J_{D_H} u$, $\phi=\overline{\J_{D_H} {v}}$, then $w\in H_0^r(\Wast;\widetilde{H}_\alpha^1(D^\Lambda_H))$. The variational \high{problem \eqref{eq:var_new}} is equivalent to
\begin{equation*}
\begin{aligned}
\int_\Wast a_{\alpha}(w({\alpha},\cdot),\phi({\alpha},\cdot))\d{\alpha}-k^2\int_{D_H} n_1^N u\overline{v}\d{ x}
=\int_\Wast F_\alpha(\phi(\alpha,\cdot))\d\alpha
\end{aligned}
\end{equation*}
for any $\phi\in H_0^{-r}(\Wast;H^1_{\alpha}(D^\Lambda_H))$, where $a_\alpha(\cdot,\cdot)$ is a sesquilinear form and $F_\alpha(\cdot)$ is an anti-linear functional defined by
\begin{eqnarray}\label{eq:vari_lhs}
&& a_{\alpha}(w,\phi):=\int_{D^\Lambda_H}\left[\nabla w\cdot\nabla \overline{\phi}-k^2 \widetilde{n}_N w\overline{\phi}\right]\d{ x}-\int_{\Gamma^\Lambda_H}T^+_{\alpha}(w)\overline{\phi}\d s,\\
\label{eq:vari-rhs}
&& F_\alpha(\phi)=-\int_{D^\Lambda_H}\left[\J_D g\right](\alpha,x)\overline{\phi}\d x,
\end{eqnarray}
and $T^+_{\alpha}: \, H^{1/2}_\alpha(\Gamma_H^\Lambda)\rightarrow H^{-1/2}_\alpha(\Gamma_H^\Lambda)$ is the ${\alpha}$-quasi-periodic Dirichlet-to-Neumann operator defined by
\begin{equation*}
T_{\alpha}^+\psi=\i\sum_{{ j}\in\Z}\sqrt{k^2-|\Lambda^* j-{\alpha}|^2}\widehat{\psi}({ j})e^{\i(\Lambda^*{ j}-{\alpha})\cdot x_1},\quad\psi=\sum_{{ j}\in\Z}\widehat{\psi}({ j})e^{\i(\Lambda^*{ j}-{\alpha})\cdot x_1}.
\end{equation*}

As $a_\alpha(\cdot,\cdot)$ is the variational form for the $\alpha$-quasi-periodic scattering problem, we only need to consider the term defined by
\begin{equation*}
 b(w,\phi)=\int_{D_H}n_1^N u\overline{v}\d x=\int_{D_H}n_1^N \left(\J_{D_H}^{-1} w\right)\overline{(\J_{D_H}^{-1}\phi)}\d x.
\end{equation*}
As $n_1^N$ is an $N\Lambda$-periodic function in $x_1$, it has a Fourier series
\begin{equation*}
 n_1^N(x_1,x_2)=\sum_{j\in\Z}\widehat{n}_j(x_2)\exp\left(2\i \frac{ j\pi} {N\Lambda}x_1\right).
\end{equation*}
To guarantee the uniform convergence of the Fourier series, we make the following assumption on $n_1$.

\begin{assumption}
 \label{asp_n1}
 $n_1$ is uniformly bounded in $D_1$. Moreover, for any fixed $x_2\in (h_0,H_1)$, $n_1(\cdot,x_2)\in C^1(\R)$ and $\frac{\partial^2 }{\partial x_1^2}n(\cdot,x_2)\in L^\infty(\R)$.
\end{assumption}
From the definition of $n_1^N$, when $n_1$ satisfies Assumption \ref{asp_n1}, $n_1^N$ also satisfies Assumption \ref{asp_n1}. 
Thus for any fixed $x_2$, the fourier series $\sum_{j\in\Z}\widehat{n}_j(x_2)\exp\left(2\i \frac{ j\pi} {N\Lambda}x_1\right)$ converges uniformly to $n_1^N(x_1,x_2)$. Moreover, the Fourier coefficient $\widehat{n}_j$ is uniformly bounded and decays at the rate of $O(|j|^{-2})$.

Inspired by \cite{Hadda2017}, the function $n_1^N$, which is $N\Lambda$-periodic in $x_1$-direction, is decomposed as $N$ quasi-periodic functions with period $\Lambda$, i.e.,
\begin{equation}\label{eq:decomp}
n_1^N(x_1,x_2)=\sum_{\ell=1}^N \exp\left(2\i \frac{\ell\pi} {N\Lambda}x_1\right)n_1^N(\ell)(x_1,x_2), 
\end{equation}
where for any $\ell=1,2,\dots,N$,
\begin{equation*}
 n_1^N(\ell)(x_1,x_2)=\sum_{j\in\Z}\widehat{n}_j(x_2)\exp\left(2\i j\pi x_1/\Lambda\right)
\end{equation*}
is a $\Lambda$-periodic function in $x_1$-direction. With Assumption \ref{asp_n1}, $n_1^N(\ell)\in L^\infty(D_H^\Lambda)$. With the representation \eqref{eq:decomp},
\begin{equation*}
\begin{aligned}
 \J_{D_H}(n_1^N u)(\alpha,x)&=C_\Lambda\sum_{j\in\Z}n_1^N(x_1+\Lambda j,x_2)u(x_1+\Lambda j,x_2)e^{-\i\alpha\Lambda j}\\
 &=C_\Lambda\sum_{j\in\Z}\left[\sum_{\ell=1}^N \exp(2\i\ell\pi (x_1+\Lambda j)/(N\Lambda))n_1^N(\ell)(x)\right]u(x_1+\Lambda j,x_2)e^{-\i\alpha\Lambda j}\\
 &=C_\Lambda \sum_{\ell=1}^N\left[\sum_{j\in\Z}u(x_1+\Lambda j,x_2)e^{-\i(\alpha-2\ell\pi/(N\Lambda))\Lambda j}\right]\exp(2\i\ell\pi x_1/(N\Lambda))n_1^N(\ell)(x)\\
 &=\sum_{\ell=1}^N \left[(\J_{D_H} u)\left(\alpha-2\ell\pi/(N\Lambda),x\right)\exp(2\i\ell\pi x_1/(N\Lambda))n_1^N(\ell)(x)\right]\\&=\sum_{\ell=1}^N \left[w\left(\alpha-2\ell\pi/(N\Lambda),x\right)\exp(2\i\ell\pi x_1/(N\Lambda))n_1^N(\ell)(x)\right].
 \end{aligned}
\end{equation*}
Then
\begin{equation*}
 b(w,\phi)=\sum_{\ell=1}^N \int_\Wast b_\ell(w,\phi)(\alpha)\d\alpha,
\end{equation*}
where
\begin{equation*}
 b_\ell(w,\phi)(\alpha)=\int_{D^\Lambda_H}w\left(\alpha-2\ell\pi/(N\Lambda),x\right)\exp(2\i\ell\pi x_1/(N\Lambda))n_1^N(\ell)(x)\overline{\phi}(\alpha,x)\d x.
\end{equation*}

Finally we arrive at the variational formulation for the transformed problem, i.e., given any $g\in H_r^{-1}(D_H)$, to find a  $w\in H_0^r(\Wast;\widetilde{H}^1_\alpha(\Omega^\Lambda_H))$ such that it satisfies
\begin{equation}\label{eq:var_Bloch}
\begin{aligned}
\int_\Wast a_\alpha(w(\alpha,\cdot),\phi(\alpha,\cdot))\d\alpha-k^2\, b(w,\phi)&=\int_\Wast F_\alpha(\phi(\alpha,\cdot))\d\alpha
\end{aligned}
\end{equation}
for any $\phi\in H_0^{-r}(\Wast;\widetilde{H}_\alpha^1(D^\Lambda_H))$.

From the arguments above, we obtain the equivalence of the weak formulation \eqref{eq:var_new} of the approximated problem and the variational problem \eqref{eq:var_Bloch}.

\begin{lemma}\label{th:equivalent}
Assume that  $g\in H^{-1}_r(D_H)$ for some $r\in[0,1)$, then $u_N\in \widetilde{H}_r^1(D_H)$ satisfies \eqref{eq:var_new} if and only if $w=\J_{D_H} u_N\in H_0^r(\Wast;\widetilde{H}_\alpha^1(D^\Lambda_H))$  satisfies \eqref{eq:var_Bloch} for $F$, which is an anti-linear functional defined in $H_0^{-r}(\Wast;H_\alpha^1(D^\Lambda_H))$, defined by \eqref{eq:vari-rhs}.
\end{lemma}

With the equivalence between \eqref{eq:var_new} and \eqref{eq:var_Bloch} in Lemma \ref{th:equivalent}, we will show the unique solvability of the variational problem \eqref{eq:var_Bloch}.

\begin{theorem}
Suppose $k^2\|n\|_\infty\leq \left\|\mathcal{B}_r(0)^{-1}\right\|^{-1}$, and Assumption \ref{asp_n1} is satisfied.  
Given any anti-linear functional $F$ on $H_0^{-r}(\Wast;H_\alpha^1(D^\Lambda_H))$ for some $r\in[0,1)$ defined by \eqref{eq:vari-rhs},  the variational problem \eqref{eq:var_Bloch} has a unique solution in $H_0^r(\Wast;\widetilde{H}^1_\alpha(D^\Lambda_H))$.
\end{theorem}

\begin{proof}

The first step is to prove the existence. From Theorem \ref{th:solv}, given any anti-linear functional $F$ defined  in  $H_0^{-r}(\Wast;H_\alpha^1(D^\Lambda_H))$, by \eqref{eq:vari-rhs}   for some $g\in H_r^{-1/2}(D_H)$, there is a unique solution $u_N\in \widetilde{H}^1_r(D_H)$ to the problem \eqref{eq:var_new}.  From Lemma \ref{th:equivalent}, $w=\J_{D_H} u_N\in H_0^r(\Wast;\widetilde{H}^1_\alpha(D^\Lambda_H))$ is a solution to the variational problem \eqref{eq:var_Bloch}. 

Then we prove the uniqueness of the solution. 
Suppose $w\in H_0^r(\Wast;\widetilde{H}^1_\alpha(D^\Lambda_H))$ is a solution to the problem \eqref{eq:var_Bloch} with  $F=0$. By choosing suitable test function $\phi$, it is easy to prove that $\J_{D_H}g=0$, thus from the property of the Bloch transform, $g=0$. Then $u_N:=\J_{D_H}^{-1}w$ is a solution to the problem \eqref{eq:var_new} with $g=0$.  From Theorem \ref{th:solv_new}, \eqref{eq:var_new} is uniquely solvable. Thus $u_N=0$, which implies that $w=0$.  The proof is finished.
\end{proof}

When $n$ and $g$ have higher regularities, the Bloch transformed field $w$ is smoother with respect to $\alpha$ (see \cite{Lechl2016b})   

\begin{theorem}\label{th:higher_reg}
Assume that $n\in C^{0,1}(D)$, i.e., it is Lipschitz continuous, $g\in H_r^{0}(D_H)$ for some $r\in[0,1)$. Then the solution $w\in H_0^r(\Wast;{H}^2_\alpha(D^\Lambda_H))$ and $u_N=\J_{D_H}^{-1} w\in {H}^2_r(D_H)$.
\end{theorem}

\begin{proof}
When $n\in C^{0,1}(D_H)$, from the definition of $n_N$, $n\in C^{0,1}(D_H)$. 
From Lemma 3.1 (a) in \cite{Lechl2009}, when $g\in L^2(D_H)\subset  H_r^{0}(D_H)$, the solution for the variational problem \eqref{eq:var_new} $u_N\in H^2(D_H)$. Then we prove that for $r\in(0,1)$, $u_N\in H^2_r(D_H)$.

Let the open cube $Q_0:=(-2,2)\times(h_0,H)$, and the translation $Q_j:=(2j,0)^\top+Q_0$, then $D_H\subset \cup_{j\in\Z}Q_j$. From \cite{Lechl2009},  there is a constant $C$ independent of $j$ such that
\begin{equation*}
 \|u_N\|_{H^2(Q_j)}\leq C\left(\|u_N\|_{H^1(Q_j)}+\|g\|_{L^2(Q_j)}\right).
\end{equation*}
To prove that $u_N\in H_r^2(D_H)$, we have to consider $H^2$-norm of the function $(1+|x|^2)^{r/2}u_N(x)$. In fact, we only need to estimate the second order partial derivatives of this function. For example, consider
\begin{equation*}
  \begin{aligned}\frac{\partial ^2 \left[(1+|x|^2)^{r/2}u_N(x)\right]}{\partial x_1^2}=\Big[r(r-2) x_1^2(1+|x|^2)^{r/2-2}+r(1+|x|^2)^{r/2-1}\Big]u_N(x)\\+2rx_1(1+|x|^2)^{r/2-1}\frac{\partial u_N(x)}{\partial x_1}+(1+|x|^2)^{r/2}\frac{\partial ^2 u_N(x)}{\partial x_1^2}.
    \end{aligned}
\end{equation*}
% then
% \begin{equation*}
% \begin{aligned}
%  \left\|\frac{\partial ^2 \left[(1+|x|^2)^{r/2}u_N(x)\right]}{\partial x_1^2}\right\|_{L^2(Q_j)}^2\leq 3\left\|\Big[r(r-2) x_1^2 (1+|x|^2)^{r/2-2}+r(1+|x|^2)^{r/2-1}\Big]u_N(x)\right\|^2_{L^2(Q_j)}\\
%  +3\left\|2rx_1(1+|x|^2)^{r/2-1}\frac{\partial u_N(x)}{\partial x_1}\right\|_{L^2(Q_j)}^2+3\left\|(1+|x|^2)^{r/2}\frac{\partial ^2 u_N(x)}{\partial x_1^2}\right\|_{L^2(Q_j)}^2
%  \end{aligned}.
% \end{equation*}
As $x_1^2(1+|x|^2)^{r/2-2}$, $(1+|x|^2)^{r/2-1}$, $x_1(1+|x|^2)^{r/2-1}$ decays when $|x_1|\rightarrow\infty$, 
\begin{eqnarray*}
&&\begin{aligned}  \left\|\Big[r(r-2) x_1^2 (1+|x|^2)^{r/2-2}+r(1+|x|^2)^{r/2-1}\Big]u_N(x)\right\|^2_{L^2(Q_j)}
\\ \leq C\|u_N\|_{L^2(Q_j)}^2\leq C\|(1+|x|^2)^{r/2}u_N\|_{H^1(Q_j)}^2;
     \end{aligned}
\\ 
&&\left\|2rx_1(1+|x|^2)^{r/2-1}\frac{\partial u_N(x)}{\partial x_1}\right\|_{L^2(Q_j)}^2
 \leq C\left\|\frac{\partial u_N(x)}{\partial x_1}\right\|_{L^2(Q_j)}^2\leq C\|(1+|x|^2)^{r/2}u_N\|_{H^1(Q_j)}^2.
\end{eqnarray*}
Use the fact that 
\[\frac{\partial (1+|x|^2)^{r/2}u_N(x)}{\partial x_j}=rx_j(1+|x|^2)^{r/2-1}u_N(x)+(1+|x|^2)^{r/2}\frac{\partial u_N(x)}{\partial x_j},\quad j=1,2,\]
we estimate the last term:
\begin{equation*}
\begin{aligned}
& \left\|(1+|x|^2)^{r/2}\frac{\partial ^2 u_N(x)}{\partial x_1^2}\right\|_{L^2(Q_j)}^2\leq \max_{x\in Q_j}\left[(1+|x|^2)^{r}\right]\left\|\frac{\partial ^2 u_N(x)}{\partial x_1^2}\right\|_{L^2(Q_j)}^2\\
 \leq& C(1+|2j|^2)^{r}\left(\|u_N\|_{H^1(Q_j)}^2+\|g\|^2_{L^2(Q_j)}\right)\\
 \leq &C(1+|2j|^2)^{r}\left(\left\|r|x|(1+|x|^2)^{-1}u_N(x)\right\|^2_{L^2(Q_j)}+\left\|(1+|x|^2)^{-r/2}\frac{\partial (1+|x|^2)^{r/2}u_N(x)}{\partial x_1}\right\|_{L^2(Q_j)}^2\right.\\
 &\qquad\Big.+\left\|(1+|x|^2)^{-r/2}\frac{\partial (1+|x|^2)^{r/2}u_N(x)}{\partial x_2}\right\|_{L^2(Q_j)}^2 +\left\|(1+|x|^2)^{-r/2}(1+|x|^2)^{r/2}g(x)\right\|^2_{L^2_{Q_j}}\Big)\\
 \leq & C(1+|2j|^2)^r\max_{x\in Q_j}\left[(1+|x|^2)^{-r}\right]\left(\left\|(1+|x|^2)^{r/2}u_N(x)\right\|^2_{H^1(Q_j)}+\left\|(1+|x|^2)^{r/2}g(x)\right\|^2_{L^2(Q_j)}\right)\\
 \leq & C(1+|2j|^2)^r(1+(|2j|-2)^2)^{-r} \left(\left\|(1+|x|^2)^{r/2}u_N(x)\right\|^2_{H^1(Q_j)}+\left\|(1+|x|^2)^{r/2}g(x)\right\|^2_{L^2(Q_j)}\right).
\end{aligned}
\end{equation*}
As when $|j|\rightarrow\infty$, $(1+|2j|^2)^r(1+(|2j|-2)^2)^{-r}$ is uniformly bounded,
\begin{equation*}
 \left\|(1+|x|^2)^{r/2}\frac{\partial ^2 u_N(x)}{\partial x_1^2}\right\|_{L^2(Q_j)}^2\leq C\left(\left\|(1+|x|^2)^{r/2}u_N(x)\right\|^2_{H^1(Q_j)}+\left\|(1+|x|^2)^{r/2}g(x)\right\|^2_{L^2(Q_j)}\right),
\end{equation*}
thus
\[\left\|\frac{\partial^2 (1+|x|^2)^{r/2}u_N(x)}{\partial x_1^2}\right\|_{L^2(Q_j)}^2\leq C\left(\left\|(1+|x|^2)^{r/2}u_N(x)\right\|^2_{H^1(Q_j)}+\left\|(1+|x|^2)^{r/2}g(x)\right\|^2_{L^2(Q_j)}\right).\]
Similarly, we can also get the similar estimations of the norm $\left\|\frac{\partial^2 (1+|x|^2)^{r/2}u_N(x)}{\partial x_2^2}\right\|_{L^2(Q_j)}$ and $\left\|\frac{\partial^2 (1+|x|^2)^{r/2}u_N(x)}{\partial x_1\partial x_2}\right\|_{L^2(Q_j)}$. Thus
\begin{equation*}
 \|(1+|x|^2)^{r/2}u_N(x)\|^2_{H^2(Q_j)}\leq C\left(\left\|(1+|x|^2)^{r/2}u_N(x)\right\|^2_{H^1(Q_j)}+\left\|(1+|x|^2)^{r/2}g(x)\right\|^2_{L^2(Q_j)}\right).
\end{equation*}

Use the fact that $D_H\subset \cup_{j\in\Z}Q_j$, we can easily obtain that
\begin{equation*}
 \|(1+|x|^2)^{r/2}u_N(x)\|^2_{H^2(D_H)}\leq C\left(\left\|(1+|x|^2)^{r/2}u_N(x)\right\|^2_{H^1(D_H)}+\left\|(1+|x|^2)^{r/2}g(x)\right\|^2_{L^2(D_H)}\right).
\end{equation*}

Thus $u_N\in H_r^2(D_H)$, then $w\in H_0^r(\Wast;H^2_\alpha(D_H^\Lambda))$. The proof is finished.

\end{proof}

When  $g$  decays faster at the infinity, the Bloch transformed field $w$ depends continuously on the quasi-periodicity parameter $\alpha$.

\begin{theorem}
If  $g\in H_r^{-1}(D_H)$ for some $r\in(1/2,1)$, then the solution $w\in H_0^r(\Wast;\widetilde{H}^1_\alpha(D^\Lambda_H))$ equivalently \high{satisfies that for all $\alpha\in\Wast$ and $\phi_\alpha\in\widetilde{H}_\alpha^1(D^\Lambda_H)$}
\begin{equation}\label{eq:var_Bloch_continuous}
a_\alpha(w(\alpha,\cdot),\phi_\alpha)-k^2\sum_{\ell=1}^N b_\ell(w(\alpha,\cdot),\phi_\alpha)=F_\alpha(\phi_\alpha).
\end{equation}

\end{theorem}

\begin{proof}
Let $\phi(t,x):=\delta_{\alpha}(t)\phi_{\alpha}(x)$, where $\delta_{\alpha}(t)$ is the Dirac Delta distribution at any fixed $\alpha\in\Wast$ and $\phi_\alpha\in\widetilde{H}_\alpha^1(D^\Lambda_H)$. As $\delta_\alpha\in H^{-r}(\Wast)$ for any $r>1/2$, $\phi\in H_0^{-r}(\Wast;\widetilde{H}_\alpha^1(D^\Lambda_H))$. Pluge the test function $\phi$ into \eqref{eq:var_Bloch}, we arrive at \eqref{eq:var_Bloch_continuous} immediately. On the other hand, if \eqref{eq:var_Bloch_continuous} holds, we can construct an orthogonal family of test functions in $L^2(\Wast;\widetilde{H}^1_\alpha(D^\Lambda_H))$ to prove that $w$ satisfies \eqref{eq:var_Bloch}. The proof is finished.

\end{proof}

\begin{remark}
In this paper, \high{ the period of } the Floquet-Bloch transform is chosen as the period of $n_2$. In fact, we can also choose the period of $n_1$ and all of the arguments are similar. 

In fact, the choice of the period is kind of arbitrary -- we can even choose the parameter which is neither the period of $n_1$ nor that of $n_2$, then the problem is treated generaly as the scattering problem with a rough layer. In this case, it is possible that the computational complexity is increased in numerical implementation.
\end{remark}

\section{Discretization  of the variational problem}

In this section, we consider the discretization with respect to $\alpha$ based on the variational formulation \eqref{eq:var_Bloch}.
\high{Define} the uniformly distributed grid points
\begin{equation*}
\alpha_N^{(j)}=\frac{2\pi j}{N\Lambda}\in\Wast,\,j=1,\dots,N,
\end{equation*}
then
\begin{equation*}
 b_j(w,z)(\alpha)=\int_{D_H^\Lambda}w\left(\alpha-\alpha_N^{(j)},x\right)e^{\i\alpha_N^{(j)}x_1}n_1^N(j)(x)\overline{z}(\alpha,x)\d x.
\end{equation*}

Let the interval be defined as $I_j=\left(\alpha_N^{(j)}-\frac{2\pi}{N\Lambda},\alpha_N^{(j)}\right]$, and $\psi_N^{(j)}(\alpha)$ be the indicator function for the interval $I_j$, i.e., $\psi_N^{(j)}(\alpha)$ takes the value $1$ in the interval $I_j$ while $0$ otherwise. Let $w$ be approximated by
\begin{equation*}
 w_N(\alpha,x):=\sum_{m=1}^N \psi_N^{(m)}(\alpha)w_N^{(m)}(x),
\end{equation*}
where $w_N^{(m)}(x)=w\left(\alpha_N^{(m)},x\right)$.
 Let $z_{j}(\alpha,x)=\psi_N^{(j)}(\alpha)z_N^{(j)}(x)$ where $z_j\in \widetilde{H}^1_{\alpha_N^{(j)}}(D^\Lambda_H)$. Pluge these two functions into the variational form \eqref{eq:var_Bloch}, then the first term becomes
\begin{equation*}
 \int_\Wast a_\alpha(w_N(\alpha,\cdot),z_{j}(\alpha,\cdot))\d\alpha=\frac{2\pi}{N\Lambda}a_N^{(j)}\left(w_N^{(j)},z_N^{(j)}\right),
\end{equation*}
where $a_N^{(j)}(\cdot,\cdot)=a_{\alpha_N^{(j)}}(\cdot,\cdot)$. 

Then we consider the second term $b(w_N,z_j)$. With the representation of $w_N(\alpha,x)$,
\begin{equation*}
\begin{aligned}
 w_N\left(\alpha-\alpha_N^{(j)},x\right)&=\sum_{m=1}^N \psi_N^{(m)}\left(\alpha-\alpha_N^{(j)}\right)w_N^{(m)}(x)\\
 &=\sum_{m=1}^N \psi_N^{(m)}\left(\alpha\right)w_N^{(m-j)}(x)
 \end{aligned}\, ,\quad\text{ where }
 w_N^{(m-j)}=\begin{cases}
              w_N^{(m-j)},\quad m>j;\\
              w_N^{m-j+N},\quad m\leq j.
             \end{cases}
\end{equation*}
Note that the new definition of $w_N^j$ (when $j\leq0$) comes from the $\Lambda^*$-periodicity of $w(\alpha,\cdot)$. Then
\begin{equation*}
\begin{aligned}
&\int_\Wast b_{m}(w,z_j)(\alpha)\d\alpha\\=&\int_\Wast\int_{D_H^\Lambda}\left[\sum_{j'=1}^N \psi_N^{(j')}\left(\alpha\right)w_N^{(j'-m)}(x)\right]e^{\i\alpha_N^{(m)}x_1}n_1^N(m)(x)\psi_N^{(j)}(\alpha)\overline{z}_N^{(j)}(x)\d x\d\alpha\\
=&\frac{2\pi}{N\Lambda}\int_{D^\Lambda_H}w_N^{(j-m)}(x)e^{\i\alpha_N^{(m)}x_1}n_1^N(m)(x)\overline{z}_N^{(j)}(x)\d x.
\end{aligned}
\end{equation*}
We also approximate the right hand side in the similar way, then
\begin{equation*}
\begin{aligned}
 \int_\Wast F_\alpha^N(z_j(\alpha,\cdot))\d\alpha&=\frac{2\pi}{N\Lambda}\left[-\int_{D_H^\Lambda}[\J_{D_H}g]\left(\alpha_N^{(j)},x\right)\overline{z}_N^{(j)}(x)\d x\right]\\
 &:=\frac{2\pi}{N\Lambda} g_j\left(z_N^{(j)}\right).
 \end{aligned}
\end{equation*}

Then we arrive at the discretized form of \eqref{eq:var_Bloch} for any fixed $j$:
\begin{equation}\label{eq:var_disc_a}
 a_N^{(j)}\left(w_N^{(j)},z_N^{(j)}\right)-k^2\sum_{m=1}^N \int_{D^\Lambda_H}e^{\i\alpha_N^{(m)}x_1}n_1^N(m)(x)w_N^{j-m}(x)\overline{z}_N^{(j)}(x)\d x=g_j\left(z_N^{(j)}\right).
\end{equation}

It is more convenient to consider functions that are periodic in $x_1$, so we define
\begin{equation*}
 \widetilde{w}_N^{(j)}(x)=e^{\i\alpha_N^{(j)} x_1}w_N^{(j)}(x),\quad \widetilde{z}_N^{(j)}(x)=e^{\i\alpha_N^{(j)} x_1}z_N^{(j)}(x),
\end{equation*}
then $\widetilde{w}_N^{(j)},\widetilde{z}_N^{(j)}\in \widetilde{H}^1_0(D_H^\Lambda)$ for any $j=1,2,\dots,N$. 
Replace $w_N^{(j)}$ and $z_N^{(j)}$ by $\widetilde{w}_N^{(j)}$ and $\widetilde{z}_N^{(j)}$ in \eqref{eq:var_disc_a}, then
\begin{equation}
 \widetilde{a}_N^{(j)}\left(\widetilde{w}_N^{(j)},\widetilde{z}_N^{(j)}\right)-k^2\sum_{m=1}^N \int_{D^\Lambda_H}\widetilde{w}_N^{(j-m)}(x)n_1^N(m)(x)\iota^{(j-m)}(x)\overline{\widetilde{z}}_N^{(j)}(x)\d x=\widetilde{g}_j\left(\widetilde{z}_N^{(j)}\right),
\end{equation}
where
\begin{eqnarray*}
 && \begin{aligned}
 \widetilde{a}_N^{(j)}\left( \widetilde{w}_N^{(j)}, \widetilde{z}_N^{(j)}\right)=&\int_{D_H^\Lambda}\left[\left(\nabla+\i\alpha_N^{(j)}{\bm e}_1\right) \widetilde{w}_N^{(j)}\cdot\left(\nabla-\i\alpha_N^{(j)}{\bm e}_1\right)\overline{\widetilde{z}}_N^{(j)}\right.\\&\left.-k^2\widetilde{n}_N  \widetilde{w}_N^{(j)}\overline{\widetilde{z}}_N^{(j)}\right]\d x-\int_{\Gamma^\Lambda_H}\widetilde{T}^+_\alpha\left( \widetilde{w}_N^{(j)}\right)\overline{\widetilde{z}}_N^{(j)}\d s,
    \end{aligned}
\\
 && \iota^{(j-m)}(x)=\begin{cases}
                          1,\quad j-m>0;\\
                          \exp(2\i\pi x_1/\Lambda),\quad j-m\leq 0;
                         \end{cases}
\\
 && \widetilde{g}_j\left(\widetilde{z}_N^{(j)}\right)=-\int_{D_H^\Lambda}[\J_{D_H}g]\left(\alpha_N^{(j)},x\right)e^{\i\alpha_N^{(j)}x_1}\overline{\widetilde{z}}_N^{(j)}(x)\d x.
\end{eqnarray*}
$\widetilde{T}_\alpha^+$ is the periodic Dirichlet-to-Neumann map defined by
\begin{equation}
 \widetilde{T}_\alpha^+\psi=\i\sum_{\ell\in\Z}\sqrt{k^2-|\Lambda^*\ell-\alpha|^2}\widehat{\psi}e^{\i\Lambda^* \ell x_1},\quad \psi=\sum_{\ell\in\Z}\widehat{\psi}e^{\i\Lambda^* \ell x_1}.
\end{equation}

Let the inverse Bloch transform of $w_N$ be denoted by $\widetilde{u}_N$, then
\begin{equation}
\left(\J_{D_H}^{-1} w_N\right)(x) =\sqrt{\frac{\Lambda}{2\pi}}\int_\Wast w_N(\alpha,x)\d\alpha=\frac{1}{N}\sqrt{\frac{2\pi}{\Lambda}}\sum_{j=1}^N e^{\i\alpha_N^{(j)}x_1}w_N^{(j)}(x),\quad x\in D^\Lambda_H.
\end{equation}

Now we have obtained the discretization of the variational problem \eqref{eq:var_new} with respect to $\alpha$, i.e., \eqref{eq:var_disc_a}. Then we  continue with the numerical scheme, i.e., to apply the finite element method for discretization with respect to the parameter $x$ in the next section.

\section{The finite element method}

In this section, we discuss a \high{Galekin} discretization of the variational formulation \eqref{eq:var_Bloch} and the finite element  method  is applied to numerical solutions. As was shown in the last section, the field $w(\alpha,\cdot)$ has been approximated by the piesewise constant function $w_N$ with respect to $\alpha$, and the discretization has been established in \eqref{eq:var_disc_a}. Thus we only need to continue with the discretization with respect to $x$.

Assume that $\mathcal{M}_h$ is a family of regular and quasi-uniform meshes (see \cite{Brenn1994}) for the periodic cell $D^\Lambda_H$,  where $0<h\leq h_0$ and $h_0$ is a small enough positive number. To obtain periodic basic functions, it is required that the nodal points on the left and right boundaries have the same heights. By omitting the nodal points on the left boundary, let $\left\{\phi_M^{(\ell)}\right\}_{\ell=1}^M$ be the family of piecewise linear and globally continuous nodal functions. For any $\phi_M^{(\ell)}$, it equals to one at the $\ell$-th point (except for the lower boundary) and zero at other nodal points. Then ${V}_h:={\rm span}\left\{\phi_M^{(\ell)}\right\}_{\ell=1}^M$ is a subspace of $\widetilde{H}_0^1(D^\Lambda_H)$. Then we define the finite element space ${X}_{N,h}$ by
\begin{equation*}
{X}_{N,h}:=\left\{v_{N,h}(\alpha,x)=\exp\left(-\i\alpha_N^{(j)} x_1\right)\sum_{j=1}^N\sum_{\ell=1}^M v_{N,h}^{(j,\ell)}\psi_N^{(j)}(\alpha)\phi_M^{(\ell)}(x):\, v_{N,h}^{(j,\ell)}\in\C\right\}.
\end{equation*}
It is easy to check that ${X}_{N,h}\subset L^2(\Wast;\widetilde{H}^1_\alpha(D^\Lambda_H))$ following \cite{Lechl2017}.  Then we  seek for a finite element solution $w_{N,h}\in {X}_{N,h}$ to the finite-dimensional (with respect to $\alpha$) problem \eqref{eq:var_disc_a} for any $v_{N,h}\in {X}_{N,h}$. Let 
\begin{equation}
 \widetilde{w}_N^{(j)}(x)=\sum_{\ell=1}^M w^{(j,\ell)}_{N,h}\phi_M^{(\ell)}(x),
\end{equation}
then
\begin{equation*}
w_{N,h}=e^{-\i\alpha_N^{(j)} x_1}\sum_{\ell=1}^N\sum_{j=1}^M w_{N,h}^{(\ell,j)}\psi_N^{(\ell)}(\alpha)\phi_M^{(j)}(x).
\end{equation*}
Let the test function $z_{j',\ell'}=e^{-\i\alpha_N^{(j')} x_1}\psi_N^{(j')}(\alpha)\phi_M^{(\ell')}(x)$ with $\widetilde{z}_N^{(j')}(x)=\phi_M^{(\ell')}(x)$. Then \eqref{eq:var_disc_a} has the discretized form
\begin{eqnarray}\label{eq:var_disc1}
&& \sum_{\ell=1}^M \delta_{j,j'}A_{\ell,\ell'}^{j}w_{N,h}^{(\ell,j)}-k^2\sum_{m=1}^N \sum_{\ell=1}^M B_{\ell,\ell'}^{m,+}w_{N,h}^{(\ell,j'-m)}=g^{j'}_{\ell'}\quad\text{ when } j>m;\\
\label{eq:var_disc2}
 && \sum_{\ell=1}^M \delta_{j,j'}A_{\ell,\ell'}^{j,j'}w_{N,h}^{(\ell,j)}-k^2\sum_{m=1}^N \sum_{\ell=1}^M B_{\ell,\ell'}^{m,-}w_{N,h}^{(\ell,j'-m+N)}=g^{j'}_{\ell'}\quad\text{ when }j\leq m;
\end{eqnarray}
where $\delta(j,j')=1$ if and only if $j=j'$, otherwise it equals to $0$. 
The coefficients are defined as follows:
\begin{eqnarray*}
&& A_{\ell,\ell'}^{j}=\widetilde{a}_N^{(j)}\left(\phi_M^{(\ell)},\phi_M^{(\ell')}\right);\\
&& B_{\ell,\ell'}^{m,+}=\int_{D_H^\Lambda}n_1^N(m)(x)\phi_M^{(\ell)}(x)\overline{\phi_M^{(\ell')}}(x)\d x\quad\text{ when }j>0;\\
&& B_{\ell,\ell'}^{m,-}=\int_{D_H^\Lambda}n_1^N(m)(x)e^{2\i\pi x_1/\Lambda}\phi_M^{(\ell)}(x)\overline{\phi_M^{(\ell')}}(x)\d x\quad\text{ when }j\leq 0;\\
&& g_{\ell'}^{j'}=\widetilde{g}_j'\left(\phi_M^{(\ell')}\right),
\end{eqnarray*}

Define the matrices and vectors as follows:
\begin{equation*}
\begin{aligned}
 &{\bm A}_j=\left(
 \begin{matrix}
  A_{1,1}^j & A_{1,2}^j & \cdots & A_{1,M}^j\\
  A_{2,1}^j & A_{2,2}^j & \cdots & A_{2,M}^j\\
  \vdots & \vdots & \ddots &\vdots\\
  A_{M,1}^j & A_{M,2}^j & \cdots & A_{M,M}^j
 \end{matrix}
 \right);\quad
  {\bm B}_m^+=\left(
 \begin{matrix}
  B_{1,1}^{m,+} & B_{1,2}^{m,+} & \cdots & B_{1,M}^{m,+}\\
  B_{2,1}^{m,+} & B_{2,2}^{m,+} & \cdots & B_{2,M}^{m,+}\\
  \vdots & \vdots & \ddots &\vdots\\
  B_{M,1}^{m,+} & B_{M,2}^{m,+} & \cdots & B_{M,M}^{m,+}
 \end{matrix}
 \right);\\
 &{\bm B}_m^-=\left(
 \begin{matrix}
  B_{1,1}^{m,-} & B_{1,2}^{m,-} & \cdots & B_{1,M}^{m,-}\\
  B_{2,1}^{m,-} & B_{2,2}^{m,-} & \cdots & B_{2,M}^{m,-}\\
  \vdots & \vdots & \ddots &\vdots\\
  B_{M,1}^{m,-} & B_{M,2}^{m,-} & \cdots & B_{M,M}^{m,-}
 \end{matrix}
 \right);\quad
 {\bm W}_j=\left( \begin{matrix}
  w_{N,h}^{1,j}\\
  w_{N,h}^{2,j}\\
  \vdots \\
  w_{N,h}^{M,j} 
 \end{matrix}
 \right);\quad
 {\bm G}_j=\left( \begin{matrix}
  g_{1}^{j}\\
  g_{2}^{j}\\
  \vdots \\
  g_{M}^{j} 
 \end{matrix}
 \right).
 \end{aligned}
\end{equation*}

Thus the discretization equation \eqref{eq:var_disc1}-\eqref{eq:var_disc2} has the form of
\begin{equation}\label{eq:dis_fem}
\Big[{\bm A}_M^N-k^2{\bm B}_M^N\Big]\left( \begin{matrix}
  {\bm W}_1\\
  {\bm W}_2\\
  \vdots \\
  {\bm W}_N 
 \end{matrix}
 \right)=\left( \begin{matrix}
  {\bm G}_1\\
  {\bm G}_2\\
  \vdots \\
  {\bm G}_N 
 \end{matrix}
 \right)
\end{equation}
where
\begin{equation*}
\begin{aligned}
& {\bm A}_M^N=\begin{pmatrix}
{\bm A}_1 & 0 &  \cdots & \cdots  & \cdots & 0\\
0 & {\bm A}_2 & 0 & & &  \vdots\\
\vdots & 0 & {\bm A}_3 & 0 &  &  \vdots\\
\vdots &  & \ddots &  \ddots & \ddots  & \vdots\\
\vdots &   & & 0 & {\bm A}_{N-1} & 0 \\
0 &   \cdots & \cdots  &\cdots  & 0 & {\bm A}_N\\
\end{pmatrix};\\
 &{\bm B}_M^N=
\begin{pmatrix}
{\bm B}_N^- & {\bm B}_{N-1}^- & {\bm B}_{N-2}^- & \cdots & \cdots & \cdots & {\bm B}_2^- & {\bm B}_1^-\\
{\bm B}_1^+ & {\bm B}_N^- & {\bm B}_{N-1}^- & {\bm B}_{N-2}^- & & & & {\bm B}_2^-\\
{\bm B}_2^+ & {\bm B}_1^+ & {\bm B}_N^- & {\bm B}_{N-1}^- & {\bm B}_{N-2}^- & & & \vdots\\
\vdots & \ddots & \ddots & \ddots & \ddots & \ddots & & \vdots\\
\vdots & & \ddots & \ddots & \ddots & \ddots & \ddots & \vdots\\
\vdots & & & {\bm B}_2^+ & {\bm B}_1^+ & {\bm B}_N^- & {\bm B}_{N-1}^- & {\bm B}_{N-2}^-\\
{\bm B}_{N-2}^+ & & & & {\bm B}_2^+ & {\bm B}_1^+ & {\bm B}_N^- & {\bm B}_{N-1}^-\\
{\bm B}_{N-1}^+ & {\bm B}_{N-2}^+ & \cdots  & \cdots & \cdots & {\bm B}_2^+ & {\bm B}_1^+ & {\bm B}_N^-\\
\end{pmatrix}.
\end{aligned}
\end{equation*}

At the end of this section, we consider the error estimate of the finite element method. Before that, we recall the Minkowski integral inequality, see Theorem 202 in \cite{Hardy1988}.

\begin{lemma}
 Suppose $(S_1,\mu_1)$ and $(S_2,\mu_2)$ are two measure spaces and $F:\,S_1\times S_2\rightarrow\R$ is measurable. Then the following inequality holds for any $p\geq 1$
 \begin{equation}\label{eq:minkowski}
  \left[\int_{S_2}\left|\int_{S_1}F(y,z)\d\mu_1(y)\right|^p\d\mu_2(z)\right]^{1/p}\leq \int_{S_1}\left(\int_{S_2}|F(y,z)|^p\d\mu_2(z)\right)^{1/p}\d\mu_1(y).
 \end{equation}

\end{lemma}

With Theorem \ref{th:solv_new} and \ref{th:err_approx}, following the proof of Theorem 9 in \cite{Lechl2017}, the convergence of the finite element method will be concluded in the following theorem.

\begin{theorem}\label{th:error}
Assume that $k^2\|n\|_\infty< \left\|\mathcal{B}_r(0)^{-1}\right\|^{-1}$,  $n\in C^{0,1}(D_H)$ and satisfies Assumption \ref{asp_n1} and $g\in H_r^{0}(D_H)$ for some $r\geq 1/2$. The linear system \eqref{eq:dis_fem} is uniquely solvable in ${X}_{N,h}$ for any $F$ defined by \eqref{eq:vari-rhs} as an anti-linear functional on  $H_0^{-r}(\Wast;H^{1}_\alpha(D^\Lambda_H))$ when $N\geq N_0$ is large enough and $0<h<h_0$ is small enough. The solution $w_{N,h}\in {X}_{N,h}$ satisfies the error estimate
\begin{equation}\label{eq:err_w}
\|w_{N,h}-w\|_{L^2(\Wast;H^\ell(D^\Lambda_H))}\leq Ch^{1-\ell}\left(N^{-r}+h\right)\|g\|_{H_r^{0}(\Gamma_H)},\quad\ell=0,1.
\end{equation}
Let $u_{N,h}:=\J^{-1}_{D_H}w_{N,h}$, then the error between $u_{N,h}$ and $u$ is bounded by
\begin{equation}\label{eq:err_u}
 \|u_{N,h}-u\|_{H^\ell(D_H^\Lambda)}\leq C \left[h^{1-\ell}(N^{-r}+h)+N^{-r}\right]\|g\|_{H_r^0(\D_H)}
\end{equation}

\end{theorem}

\begin{proof}
 From Theorem 9 in \cite{Lechl2017}, \eqref{eq:err_w} is easily obtained, thus we only need to prove \eqref{eq:err_u}. From Theorem \ref{th:err_approx}, as $\|u_N-u\|_{H^1(D_H)}\leq C(N\Lambda/4)^{-r}\|u\|_{H_r^0(D_H)}$, we only need to consider the difference between $u_{N,h}$ and $u_N$. From the definition of the two functions, 
 \begin{equation*}
   u_{N,h}-u_N=\sqrt{\frac{\Lambda}{2\pi}}\int_\Wast\left(w_{N,h}-w\right)(\alpha,\cdot)\d \alpha.
 \end{equation*}
With the help of \eqref{eq:minkowski}, we estimate the $L^2(D^\Lambda_H)$ norm of the above function:
\begin{equation*}
 \begin{aligned}
  \left\|u_{N,h}-u_N\right\|_{L^2(D^\Lambda_H)}&=\sqrt{\frac{\Lambda}{2\pi}}\left[\int_{D^\Lambda_H}\left|\int_\Wast (w_{N,h}-w)(\alpha,x)\d\alpha\right|^2\d x\right]^{1/2}\\
  &\leq \sqrt{\frac{\Lambda}{2\pi}}\int_\Wast\left(\int_{D^\Lambda_H}|(w_{N,h}-w)(\alpha,x)|^2\d x\right)^{1/2}\d \alpha\\
  &=\sqrt{\frac{\Lambda}{2\pi}}\|w_{N,h}-w\|_{L^2(\Wast;L^2(D^\Lambda_H))}\\
  &\leq Ch(N^{-r}+h)\|g\|_{H_r^0(\Gamma_H)}.
 \end{aligned}
\end{equation*}
Similarly, we can also estimate the $H^1-$norm between $u_{N,h}$ and $u_N$. Thus
\begin{equation*}
 \left\|u_{N,h}-u_N\right\|_{H^1(D^\Lambda_H)}\leq C(N^{-r}+h)\|g\|_{H_r^0(D_H)}.
\end{equation*}
Then we finally arrive at
\begin{equation*}
\begin{aligned}
 \left\|u_{N,h}-u\right\|_{H^\ell(D^\Lambda_H)}&\leq \left\|u_{N,h}-u_N\right\|_{H^\ell(D^\Lambda_H)}+\left\|u_N-u\right\|_{H^\ell(D^\Lambda_H)}\\
 &\leq C \left[h^{1-\ell}(N^{-r}+h)+N^{-r}\right]\|g\|_{H_r^0(\D_H)}.
 \end{aligned}
\end{equation*}
The proof is finished.
 
\end{proof}

\section{Numerical examples}

In this section, we present eight examples to illustrate the convergence result of the numerical algorithm. We choose two different groups of refractive indexes $(n_1,n_2)$, both of which are embedded in the half domain above the line $\R\times\{1\}$:\\
{\bf Group 1:}
\begin{eqnarray*}
&& n_1^{(1)}=0.1\sin\left( x_1/\sqrt{2}\right)\mathcal{X}_{0.3}(|x_2-1.5|);\\ &&n_2^{(1)}=0.25\sin(x_1)\mathcal{X}_{0.3}(|x_2-2.5|).
\end{eqnarray*}
{\bf Group 2:}
\begin{eqnarray*}
&& n_1^{(2)}=-0.25\mathcal{X}_{4}(|x_1|)\mathcal{X}_{0.3}(|x_2-1.5|)\text{ and is extended $15$-periodically in $x_1$-direction}
;\\
&&n_2^{(2)}=0.25\mathcal{X}_{0.3}\left(\sqrt{x_1^2+(x_2-2.5)^2}\right)\text{ and is extended $2\pi$-periodically in $x_1$-direction}.
\end{eqnarray*}
Note that $\mathcal{X}_{a}$ is a $C^1$-continuous function with a bounded second order derivative defined in $[0,\infty)$ as:
\begin{equation*}
 \mathcal{X}_a(t)=\begin{cases}
1,\quad\text{ when }t\leq a/2;\\
-\frac{4(a-x)^2(a-4x)}{a^3},\quad\text{ when }a/2<t<a;\\
0,\quad\text{ when }t\geq a.
                  \end{cases}
\end{equation*}
Thus both $n_1^{(1)}$ and $n_1^{(2)}$ satisfy Assumption \ref{asp_n1}. Moreover, $n$ also belongs to the space $C^{0,1}(D)$, as was assumed in Theorem \ref{th:higher_reg}.  Both $n_2^{(1)}$ and $n_2^{(2)}$ are $2\pi$-periodic functions in $x_1$-direction supported in the strip $\R\times(2,3)$; while both the function $n_1^{(1)}$ and $n_1^{(2)}$ are supported in $\R\times(1,2)$. Moreover, $n_1^{(1)}$ is $2\sqrt{2}\pi$-periodic while $n_1^{(2)}$ is $15$-periodic.

In numerical examples, the following parameters are fixed:
\begin{equation*}
\Lambda=2\pi,\,\Lambda^*=1,\, H=3,\, H_1=2,\, h_0=1.
\end{equation*}

In the numerical implementation, we use very large number of points to approximate the Fourier series of $n_1^N(x_1,x_2)$, i.e., $1000 N$ points in $x_1$-direction and $1000$ points in $x_2$-direction, where $N$ is the number of uniform subintervals introduced in Section 3. Then we use the $-8N$-th to $8N$-th coefficients to construct the decomposition \eqref{eq:decomp}, and use the ``pchip'' interpolation in MATLAB to obtain values on mesh points. As the Fourier coefficients decays at the rate of $O(N^{-2})$, we assume that the error brought by this approximation is small enough to be ignored.

\subsection{Numerical examples with exact solutions}

Recall the half space Green's function
\begin{equation*}
G(x,y)=\frac{\i}{4}\left[H_0^{(1)}(k|x-y|)-H_0^{(1)}(k|x-y'|)\right],
\end{equation*}
where $y=(y_1,y_2)^\top$ and $ y'=(y_1,-y_2)^\top$. Furthermore, we assume that $0<y_2<h_0$, thus the point source is located in the upper half space $\R\times(0,+\infty)$ and below $\Gamma_{h_0}$ (see Figure \ref{sample_point}). It is easy to check that, $G(x,y)\in H^1_r(D_H)$ for any $|r|<1$.

\begin{figure}[H]
\centering
\includegraphics[width=15cm]{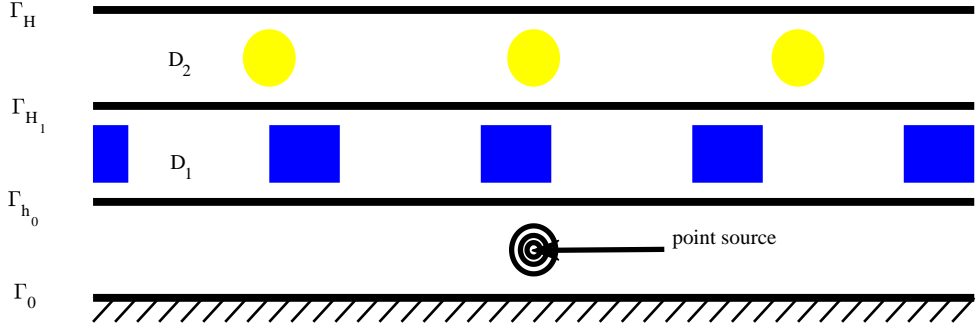}
\caption{Locations of the periodic layers and point sources.}
\label{sample_point}
\end{figure}

 For a fixed point $y$, $G(\cdot,y)$ solves the following equations:
\begin{eqnarray*}
 &&\Delta u+k^2(1+ n) u=g\quad\text{ in }D_H;\\
 && u=G(\cdot,y)\quad\text{ on }\Gamma_{h_0};\\
 && \frac{\partial u}{\partial x_2}-T^+ u=0\quad\text{ on }\Gamma_H;
\end{eqnarray*}
where 
\begin{eqnarray*}
g(x)=k^2 n(x) G(x,y)\quad \text{ in }D_H.
\end{eqnarray*}
From the property of $G(x,y)$ and $n$, $g\in H^0_r(D_H)$.

\begin{remark}
There is a little difference between the numerical examples in this subsection (and also the next subsection)  and the original problem \eqref{eq:var_origional}, as the homogeneous boundary conditions on $\Gamma_{h_0}$ or $\Gamma_H$ are changed. For numerical implementation, we can modify the algorithm with similar technique introduced in \cite{Lechl2017}. For error estimate, we can modify the problem to get an equivalent one in the form of \eqref{eq:var_origional}. Let $\widetilde{u}:=u-u_0$, where $u_0=G(\cdot,y)$ on $\Gamma_{h_0}$, and is extended to a smooth function with a support in $\R\times[h_0,H']$ for some $h_0<H'<H$. Then $u_0$ satisfies \eqref{eq:var_origional} with $g$ replaced with $g-(\Delta+k^2(1+n))u_0$. The regularity of the right hand side is decided by both $g$ and $u_0$, and we can carry out the error estimation in the same way as \eqref{eq:var_origional}. For numerical examples in the next subsection, similar technique can be employed as well.
\end{remark}

We choose one fixed Green's function in this subsection located at the point
\begin{equation*}
P=(0.5,0.4). 
\end{equation*}
The numerical scheme is carried out for the mesh size $h$ is chosen as $0.64,\,0.32,\,0.16,\,0.08$ for $k=6$ and $0.16,\,0.08,\,0.04,\,0.02$ for $k=1$, and the parameter $N$ is taken as $10,20,40,80$. Then the following four examples are considered for different $h$ and $N$, and the relative $L^2$-errors on $\Gamma_H$, defined by
\begin{equation*}
err=\frac{\|u_{N,h}-u\|_{L^2(\Gamma^\Lambda_H)}}{\|u\|_{L^2(\Gamma^\Lambda_H)}}
\end{equation*}
are listed in Table \ref{surf1k1}-\ref{surf2k2}, where the exact solution is $u=G(\cdot,y)$.\\

\noindent
{\bf Example 1.} The wave number $k=1$, the refractive indexes are defined by $n_1^{(1)}$ and $n_2^{(1)}$, the relative errors are listed in Table \ref{surf1k1}.\\

\noindent
{\bf Example 2.} the wave number $k=6$, the refractive indexes are defined by $n_1^{(1)}$ and $n_2^{(1)}$, the relative errors are listed in Table \ref{surf1k2}.\\

\noindent{\bf Example 3.} The wave number $k=1$, the refractive indexes are defined by $n_1^{(2)}$ and $n_2^{(2)}$, the relative errors are listed in Table \ref{surf2k1}.\\

\noindent{\bf Example 4.} the wave number $k=6$, the refractive indexes are defined by $n_1^{(2)}$ and $n_2^{(2)}$, the relative errors are listed in Table \ref{surf2k2}.\\

\begin{table}[htb]
\centering
\caption{Relative $L^2$-errors for Example 1.}\label{surf1k1}
\begin{tabular}
{|p{1.8cm}<{\centering}||p{2cm}<{\centering}|p{2cm}<{\centering}
 |p{2cm}<{\centering}|p{2cm}<{\centering}|p{2cm}<{\centering}|}
\hline
  & $h=0.64$ & $h=0.32$ & $h=0.16$ & $h=0.08$\\
\hline
\hline
$N=10$&$6.3$E$-02$&$4.5$E$-02$&$4.6$E$-02$&$4.7$E$-02$\\
\hline
$N=20$&$5.1$E$-02$&$2.1$E$-02$&$1.6$E$-02$&$1.7$E$-02$\\
\hline
$N=40$&$4.8$E$-02$&$1.5$E$-02$&$6.6$E$-03$&$5.9$E$-03$\\
\hline
$N=80$&$4.8$E$-02$&$1.4$E$-02$&$3.9$E$-03$&$2.2$E$-03$\\
\hline
\end{tabular}
\end{table}

\begin{table}[htb]
\centering
\caption{Relative $L^2$-errors for Example 2.}\label{surf1k2}
\begin{tabular}
{|p{1.8cm}<{\centering}||p{2cm}<{\centering}|p{2cm}<{\centering}
 |p{2cm}<{\centering}|p{2cm}<{\centering}|p{2cm}<{\centering}|}
\hline
  & $h=0.16$ & $h=0.08$ & $h=0.04$ & $h=0.02$\\
\hline
\hline
$N=10$&$3.8$E$-01$&$1.4$E$-01$&$1.1$E$-01$&$1.1$E$-01$\\
\hline
$N=20$&$3.8$E$-01$&$1.1$E$-01$&$4.9$E$-02$&$4.2$E$-02$\\
\hline
$N=40$&$3.8$E$-01$&$1.0$E$-01$&$3.1$E$-02$&$1.7$E$-02$\\
\hline
$N=80$&$3.8$E$-01$&$1.0$E$-01$&$2.8$E$-02$&$8.9$E$-03$\\
\hline
\end{tabular}
\end{table}

\begin{table}[htb]
\centering
\caption{Relative $L^2$-errors for Example 3.}\label{surf2k1}
\begin{tabular}
{|p{1.8cm}<{\centering}||p{2cm}<{\centering}|p{2cm}<{\centering}
 |p{2cm}<{\centering}|p{2cm}<{\centering}|p{2cm}<{\centering}|}
\hline
  & $h=0.64$ & $h=0.32$ & $h=0.16$ & $h=0.08$\\
\hline
\hline
$N=10$&$6.5$E$-02$&$4.6$E$-02$&$4.6$E$-02$&$4.7$E$-02$\\
\hline
$N=20$&$5.2$E$-02$&$2.2$E$-02$&$1.7$E$-02$&$1.7$E$-02$\\
\hline
$N=40$&$4.9$E$-02$&$1.6$E$-02$&$6.7$E$-03$&$5.9$E$-03$\\
\hline
$N=80$&$4.8$E$-02$&$1.5$E$-02$&$4.0$E$-03$&$2.2$E$-03$\\
\hline
\end{tabular}
\end{table}

\begin{table}[htb]
\centering
\caption{Relative $L^2$-errors for Example 4.}\label{surf2k2}
\begin{tabular}
{|p{1.8cm}<{\centering}||p{2cm}<{\centering}|p{2cm}<{\centering}
 |p{2cm}<{\centering}|p{2cm}<{\centering}|p{2cm}<{\centering}|}
\hline
  & $h=0.16$ & $h=0.08$ & $h=0.04$ & $h=0.02$\\
\hline
\hline
$N=10$&$4.0$E$-01$&$1.2$E$-01$&$9.1$E$-02$&$1.1$E$-01$\\
\hline
$N=20$&$4.0$E$-01$&$1.1$E$-01$&$4.2$E$-02$&$3.9$E$-02$\\
\hline
$N=40$&$4.0$E$-01$&$1.1$E$-01$&$3.1$E$-02$&$1.6$E$-02$\\
\hline
$N=80$&$4.0$E$-01$&$1.1$E$-01$&$2.9$E$-02$&$8.9$E$-03$\\
\hline
\end{tabular}
\end{table}

From the relative errors listed in Table \ref{surf1k1}-\ref{surf2k2}, the errors decrease when $N$ gets larger and $h$ gets smaller, thus it is shown  that the numerical method converges as $N\rightarrow\infty$ and $h\rightarrow 0^+$. When the wave number is relatively small, i.e., $k=1$, the error brought by $N$ is dominant, thus the error brought by $h$ could be ignored for small enough $h$, e.g., $h=0.02,\,0.04$ in Example 1 and 3, see the last columns in Table \ref{surf1k1} and \ref{surf2k1}. From Figure \ref{eg} (a), the convergence rate with respect to $N$ is about $O(N^{-1.5})$, which is even better than expected. On the other hand, when the wave number is large, i.e., $k=6$, the dominant part of the relative error is brought by $h$, and the error brought by $N$ could ignored when $N$ is large enough, e.g., $N=80$ in Example 2 and 4, see the last lines in Table \ref{surf1k2} and \ref{surf2k2}. From Figure \ref{eg} (b), the convergence rate with respect to $h$ could reach $O(h^{1.8})$, which is almost as high as expected.  Thus the convergence result proved in Theorem \ref{th:error} is illustrated. Moreover, as the convergence rate with respect to $N$  is higher than expected, it is expected that the numerical algorithm may be improved with similar technique introduced in \cite{Zhang2017e}.

% \begin{figure}[H]
% \centering
% \includegraphics[width=7cm]{eg1}
% \caption{The relative $L^2$-errors for {Example} 1 with $h=0.02$ plotted in logarithmic scale over $N$.}
% \label{eg1}
% \end{figure}
% 
% \begin{figure}[H]
% \centering
% \includegraphics[width=7cm]{eg2}
% \caption{The relative $L^2$-errors for {Example} 2 with $N=80$ plotted in logarithmic scale over $h$.}
% \label{eg2}
% \end{figure}

\begin{figure}[tttttt!!!b]
\centering
\begin{tabular}{c c}
\includegraphics[width=0.45\textwidth]{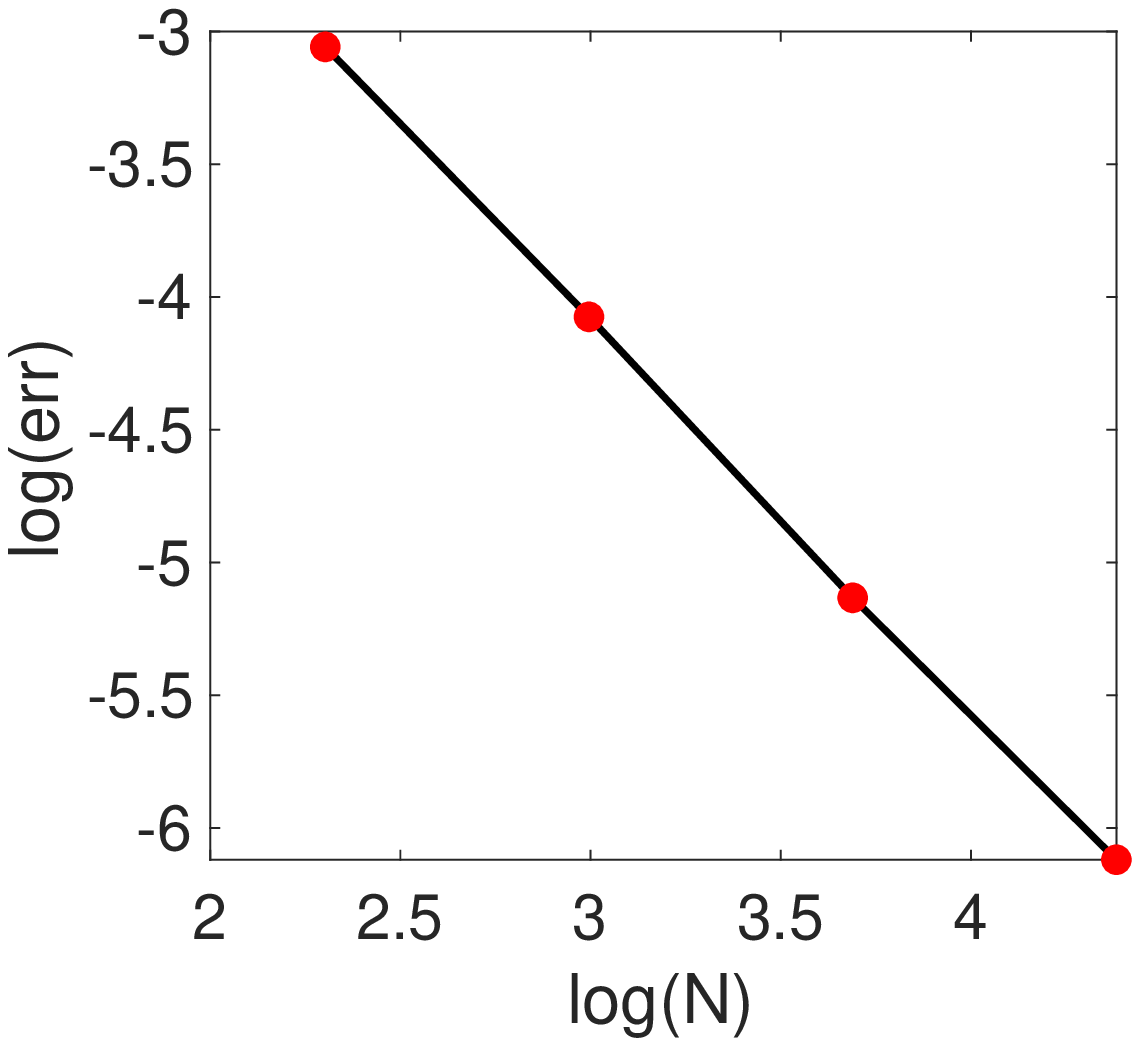} 
& \includegraphics[width=0.45\textwidth]{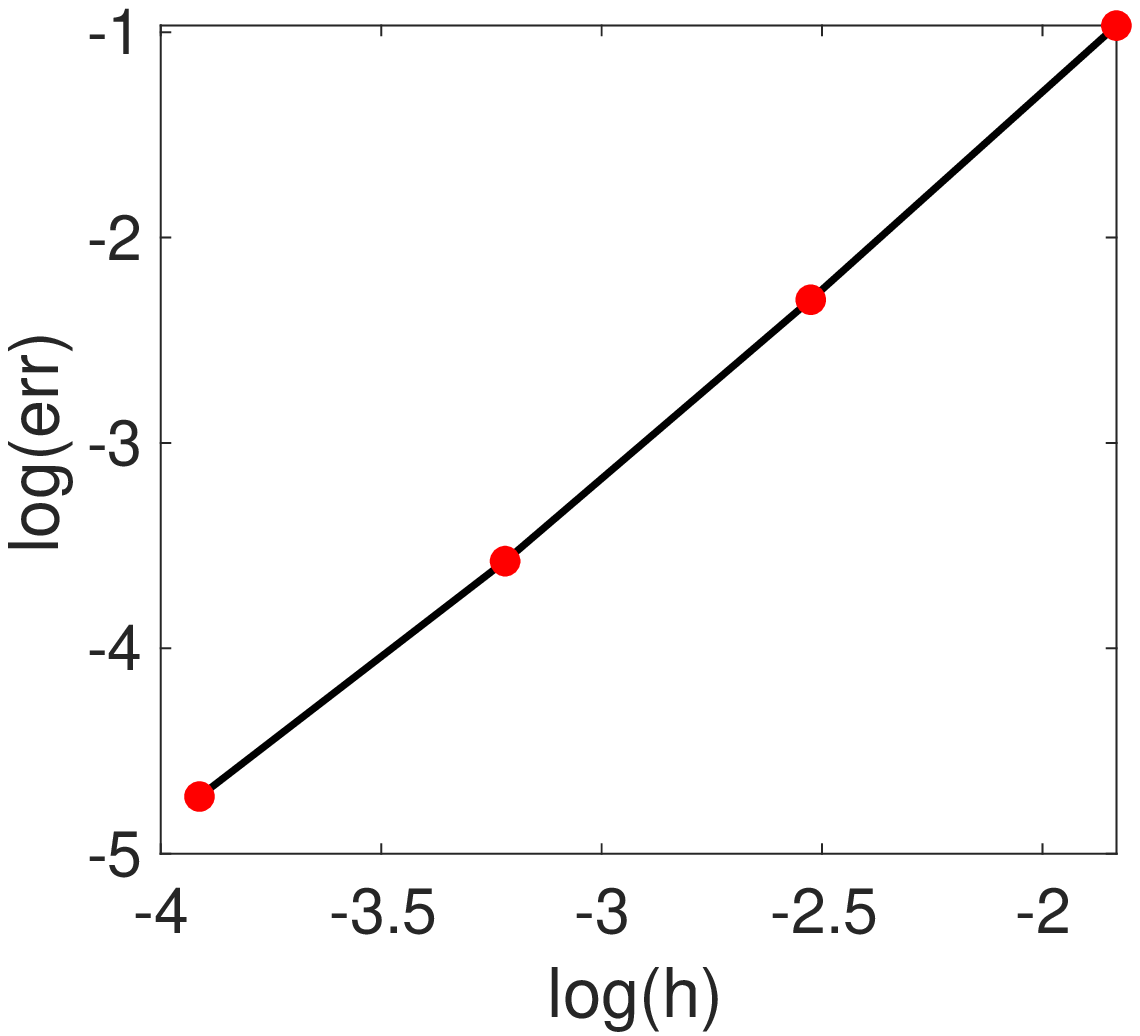} \\[-0cm]
(a) & (b) 
\end{tabular}
\caption{(a): The relative $L^2$-errors for {Example} 1 with $h=0.02$ plotted in logarithmic scale over $N$. (b): The relative $L^2$-errors for {Example} 2 with $N=80$ plotted in logarithmic scale over $h$.}
\label{eg}
\end{figure}

\subsection{Numerical example with non-exact solutions}

In this subsection, the incident field $G(x,y)$ is located above $D_H$, i.e., $y=(\pi,4)^\top$. Thus $u$ satisfies the following equations
\begin{eqnarray*}
 &&\Delta u+k^2(1+ n) u=0\quad\text{ in }D_H;\\
 && u=0\quad\text{ on }\Gamma_{h_0};\\
 && \frac{\partial u}{\partial x_2}-T^+ u=f\quad\text{ on }\Gamma_H;
\end{eqnarray*}
where 
\begin{eqnarray*}
f=\frac{\partial G(\cdot,y)}{\partial x_2}-T^+ G(\cdot,y)\quad \text{ on }\Gamma_H.
\end{eqnarray*}
From the property of $G(\cdot,y)$, $f\in H^{-1/2}_r(\Gamma_H)$ for any $|r|<1$.

As no exact solution is known for the refractive indexes $\left(n_1^{(1)},n_2^{(1)}\right)$ and $\left(n_1^{(2)},n_2^{(2)}\right)$, we can only use finer meshes to produce an ``exact solution''. We set the parameters for the finer meshes to be $h=0.01$ and $N=160$, and let the solution with respect to these meshes be the ``exact solution'' $u$. We set Example 5-8 as follows:\\

\noindent
{\bf Example 5.} The wave number $k=1$, the refractive indexes are defined by $n_1^{(1)}$ and $n_2^{(1)}$. The relative errors with $h=0.16,\,0.08,\,0.04,\,0.02$ and $N=10,\,20,\,40,\,80$ are listed in Table \ref{surf1k1s}.\\

\noindent
{\bf Example 6.} the wave number $k=6$, the refractive indexes are defined by $n_1^{(1)}$ and $n_2^{(1)}$. The relative errors with $0.08,\,0.04,\,0.02$ and $N=10,\,20,\,40,\,80$ are listed in Table \ref{surf1k2s}.\\

\noindent{\bf Example 7.} The wave number $k=1$, the refractive indexes are defined by $n_1^{(2)}$ and $n_2^{(2)}$. The relative errors with $h=0.16,\,0.08,\,0.04,\,0.02$ and $N=10,\,20,\,40,\,80$ are listed in Table \ref{surf2k1s}.\\

\noindent{\bf Example 8.} the wave number $k=6$, the refractive indexes are defined by $n_1^{(2)}$ and $n_2^{(2)}$. The relative errors with $0.08,\,0.04,\,0.02$ and $N=10,\,20,\,40,\,80$ are listed in Table \ref{surf2k2s}.\\

From Table \ref{surf1k1s}-\ref{surf2k2s}, we can conclude similar convergence results as in the last subsection. However, due to the limited memory of our computers, we can not use finer meshes to produce better ``exact solutions'', which results in worse relative errors compare with those in Table \ref{surf1k1}-\ref{surf2k2}. However, although the numerical results are not as good as Example 1-4, they still shows that our algorithm converges as $h\rightarrow 0$ and $N\rightarrow \infty$.

\begin{table}[htb]
\centering
\caption{Relative $L^2$-errors for Example 5.}\label{surf1k1s}
\begin{tabular}
{|p{1.8cm}<{\centering}||p{2cm}<{\centering}|p{2cm}<{\centering}
 |p{2cm}<{\centering}|p{2cm}<{\centering}|p{2cm}<{\centering}|}
\hline
  & $h=0.16$ & $h=0.08$ & $h=0.04$ & $h=0.02$\\
\hline
\hline
$N=10$&$9.9$E$-02$&$9.3$E$-02$&$8.9$E$-02$&$8.9$E$-02$\\
\hline
$N=20$&$6.0$E$-02$&$4.2$E$-02$&$3.2$E$-02$&$3.1$E$-02$\\
\hline
$N=40$&$5.4$E$-02$&$3.0$E$-02$&$1.4$E$-02$&$1.1$E$-02$\\
\hline
$N=80$&$5.3$E$-02$&$2.8$E$-02$&$1.1$E$-02$&$4.6$E$-03$\\
\hline
\end{tabular}
\end{table}

\begin{table}[htb]
\centering
\caption{Relative $L^2$-errors for Example 6.}\label{surf1k2s}
\begin{tabular}
{|p{1.8cm}<{\centering}||p{2cm}<{\centering}|p{2cm}<{\centering}
 |p{2cm}<{\centering}|p{2cm}<{\centering}|p{2cm}<{\centering}|}
\hline
   & $h=0.08$ & $h=0.04$ & $h=0.02$\\
\hline
\hline
$N=10$&$6.7$E$-01$&$2.4$E$-01$&$2.5$E$-01$\\
\hline
$N=20$&$6.1$E$-01$&$1.1$E$-01$&$9.4$E$-02$\\
\hline
$N=40$&$6.0$E$-01$&$9.1$E$-02$&$4.0$E$-02$\\
\hline
$N=80$&$6.0$E$-01$&$9.1$E$-02$&$3.0$E$-02$\\
\hline
\end{tabular}
\end{table}

\begin{table}[htb]
\centering
\caption{Relative $L^2$-errors for Example 7.}\label{surf2k1s}
\begin{tabular}
{|p{1.8cm}<{\centering}||p{2cm}<{\centering}|p{2cm}<{\centering}
 |p{2cm}<{\centering}|p{2cm}<{\centering}|p{2cm}<{\centering}|}
\hline
  & $h=0.16$ & $h=0.08$ & $h=0.04$ & $h=0.02$\\
\hline
\hline
$N=10$&$1.0$E$-01$&$9.7$E$-02$&$9.3$E$-02$&$9.3$E$-02$\\
\hline
$N=20$&$6.2$E$-02$&$4.3$E$-02$&$3.3$E$-02$&$3.2$E$-02$\\
\hline
$N=40$&$5.5$E$-02$&$3.1$E$-02$&$1.5$E$-02$&$1.1$E$-02$\\
\hline
$N=80$&$5.5$E$-02$&$2.9$E$-02$&$1.1$E$-02$&$4.7$E$-03$\\
\hline
\end{tabular}
\end{table}

\begin{table}[htb]
\centering
\caption{Relative $L^2$-errors for Example 8.}\label{surf2k2s}
\begin{tabular}
{|p{1.8cm}<{\centering}||p{2cm}<{\centering}|p{2cm}<{\centering}
 |p{2cm}<{\centering}|p{2cm}<{\centering}|p{2cm}<{\centering}|}
\hline
  & $h=0.08$ & $h=0.04$ & $h=0.02$\\
\hline
\hline
$N=10$&$8.3$E$-01$&$6.3$E$-01$&$7.5$E$-01$\\
\hline
$N=20$&$6.0$E$-01$&$2.4$E$-01$&$2.8$E$-01$\\
\hline
$N=40$&$5.7$E$-01$&$1.2$E$-01$&$9.5$E$-02$\\
\hline
$N=80$&$5.7$E$-01$&$1.1$E$-01$&$4.1$E$-02$\\
\hline
\end{tabular}
\end{table}

\subsection{computational complexity}

At the end of this paper, we would like to make a comment on the computational complexity, especially the comparison with the 
classic finite section method. The variational formulation of the finite section method is, to find the solution $u_N\in\widetilde{H}^1\left(\widetilde{D}_H^N\right)$ such that for any $v\in\widetilde{H}^1\left(\widetilde{D}_H^N\right)$,
\[\int_{\widetilde{D}_H^N}[\grad u_N\cdot\grad\overline{v}-k^2 (1+n_N)u_n\overline{v}]\d x-\int_{\Gamma_H^N}T^+\left[u_N\Big|_{\Gamma_H}\right]\overline{v}\d s=-\int_{\widetilde{D}_H^N}g\overline{v}\d x,\]
where $\widetilde{D}_H^N=D_H\cap[-N\Lambda/2,N\Lambda/2]\times\R$, $\Gamma_H^N=[-N\Lambda/2,N\Lambda/2]\times\{H\}$. The main difference between  computational complexities of the new method and the finite section method is the evaluation of the term with gradients. Suppose there are $M$ mesh points in one periodic cell $D_H^\Lambda$, then there are $NM$ points in the domain $\widetilde{D}_H^N$. Thus for the finite section method, the evaluation of the term $\grad u_N\cdot\grad \overline{v}$ will be carried out $NM$ times. However, for the Floquet-Bloch transform based method, it is only evaluated for the points in $D_H^\Lambda$, thus is carried out only $M$ times (see the formulation of the matrix $A_M^N$). For the second term, from the formulation of $B_M^N$, the computational complexity is almost the same for both methods. For piecewise linear basic functions in triangular meshes, the computational complexity of the evaluation of 
\[\grad u_N\cdot\grad\overline{v}=\frac{\partial u_N}{\partial x_1}\frac{\partial \overline{v}}{\partial x_1}+\frac{\partial u_N}{\partial x_2}\frac{\partial \overline{v}}{\partial x_2}\]
is twice as much as the term $(1+n_N)u_n\overline{v}$, and the value of the integral on $\Gamma_H^N$ is ignorable as it is a one-dimensional problem. Thus roughly speaking, the computational complexity of the finite section method is $C(2+1)NM$ while the value of the new method is $C(2+N)M$, where $C$ is a bounded value independent of methods. Thus when $N$ is large enough, the computational complexity of the new method is about $1/3$ as much as the finite section method. Moreover, the new method also saves a lot of time and space in the evaluation and the storage of coefficients of the basis functions and their derivatives, thus the new method reduces the computational complexity significantly.

\section*{Appendix: The Floquet-Bloch transform}

The main tool used in this paper is the Floquet-Bloch transform. In the Appendix, we recall the definition and some basic properties of the Bloch transform in periodic domains in $\R^2$ (for details see \cite{Lechl2016}).

Suppose $\Omega\subset\R^2$ is $\Lambda$-periodic in $x_1$ - direction, i.e., for any ${ x}=(x_1,x_2)^\top\in\Omega$, the translated point $(x_1+\Lambda j,x_2)\in\Omega,\,\forall{ j}\in\Z$. %Moreover, assume there is an $L>0$ such that $\Omega\subset\R\times[-L,L]$. 
Define one periodic cell by $\Omega^\Lambda:=\Omega\cap\left[\W\times\R\right]$ where $\W=(-\Lambda/2,\Lambda/2]$. For any $\phi\in C_0^\infty(\Omega)$, define the (partial)  Bloch transform in $\Omega$, i.e., $\J_{\Omega}$, of $\phi$ as
\begin{equation*}
\left(\J_\Omega\phi\right)({\alpha},{ x})=C_\Lambda\sum_{{ j}\in\Z}\phi\left({ x}+\left(\begin{matrix}
\Lambda { j}\\0
\end{matrix}\right)\right)e^{-\i{\alpha}\cdot\Lambda{ j}},\quad {\alpha}\in\R ,\,{ x}\in\Omega^\Lambda,
\end{equation*}
where $C_\Lambda$ is a constant defined by $C_\Lambda:=\left[\frac{\Lambda}{2\pi}\right]^{1/2}$.

\begin{remark}
The periodic domain $\Omega$ is not required to be bounded in $x_2$-direction.  
\end{remark}

We can also define the weighted Sobolev  space on the unbounded domain $\Omega$ by
\begin{equation*}
H_r^s(\Omega):=\left\{\phi\in \mathcal{D}'(\Omega):\,(1+|{ x}|^2)^{r/2}\phi({ x})\in H^s(\Omega)\right\}.
\end{equation*}
For any $\ell\in\N$, $s\in\R$, we can also define the following Hilbert space by
\begin{equation*}
H^\ell(\Wast;H^s(\Omega^\Lambda)):=\left\{\psi\in\mathcal{D}'(\Wast\times\Omega^\Lambda):\,\sum_{m=0}^\ell\int_\Wast\left\|\partial^m_{\alpha}\psi({\alpha},\cdot)\right\|\d{\alpha}<\infty\right\}.
\end{equation*}
From interpolation and duality arguments, we can extend the definition of the space $H_0^r(\Wast;H_\alpha^s(\Omega^\Lambda))$ for any $r,\,s\in\R$. The following properties for the $d$-dimensional (partial) Bloch transform $\J_\Omega$ is also proved in \cite{Lechl2016}.

\begin{theorem}\label{th:Bloch_property}
The Bloch transform $\J_\Omega$ extends to an isomorphism between $H_r^s(\Omega)$ and $H_0^r(\Wast;H_\alpha^s(\Omega^\Lambda))$ for any $s,r\in\R$. Its inverse has the form of
\begin{equation*}
(\J^{-1}_\Omega\psi)\left({ x}+\left(\begin{matrix}
\Lambda { j}\\0
\end{matrix}\right)\right)=C_\Lambda\int_\Wast \psi({\alpha},{ x})e^{\i{\alpha}\cdot\Lambda{ j}}\d{\alpha},\quad x_1\in\Omega^\Lambda,\,{ j}\in\Z,
\end{equation*}
and the adjoint operator $\J^*_\Omega$ with respect to the scalar product in $L^2(\Wast;L^2(\Omega^\Lambda))$ equals to the inverse $\J^{-1}_\Omega$. Moreover, when $r=s=0$, the Bloch transform $\J_\Omega$ is an isometric isomorphism.
\end{theorem}

Another important property of the Bloch transform is \high{that it} commutes with partial derivatives, see \cite{Lechl2016}. If $u\in H_r^n(\Omega)$ for some $n\in\N$, then for any ${ \gamma}=(\gamma_1,\gamma_2)\in\N^2$ with $|\gamma|=|\gamma_1|+|\gamma_2|\leq N$,
\begin{equation*}
\partial^{ \gamma}_{  x} \left(\J_\Omega u\right)({ \alpha},{  x})=\J_\Omega[\partial^{ \gamma} u]({ \alpha},{  x}).
\end{equation*}

\begin{remark}
The definition of the partial Bloch transform could also be extended to other periodic domains, for example, periodic hyper-surfaces. If $\Gamma$ is a $\Lambda$-periodic surface defined in $\R^2$, then we can define $\J_\Gamma$ in the same way, and obtain the same properties. In this paper, we will denote the Bloch transform $\J_X$ by the partial Bloch transform in the domain $X\subset\R^2$, which is periodic with respect to $x_1$-direction.
\end{remark}

\begin{remark}\label{rem:Four}
There is an alternative definition for the space $H_0^r(\Wast;X_{ \alpha})$, where $X_{ \alpha}$ is a family of Hilbert spaces that are ${ \alpha}$-quasi-periodic in $\widetilde{  x}$. Let 
\begin{equation*}
\phi_{\Lambda^*}^{({  j})}({ \alpha})=\frac{1}{\sqrt{|\Lambda^*|}}e^{\i{ \alpha}\cdot\Lambda {  j}},\,{  j}\in\Z
\end{equation*}
be a complete orthonormal system in $L^2(\Wast)$, then any function $\psi\in \mathcal{D}'(\Wast\times\Omega^\Lambda)$ has a Fourier series
\begin{equation*}
\psi({ \alpha},{  x})=\frac{1}{\sqrt{\Lambda^*}}\sum_{{ \ell}\in\Z}\hat{\psi}_{\Lambda^*}({ \ell},{  x})e^{\i{ \alpha}\cdot\Lambda{ \ell}},
\end{equation*}
where $\hat{\psi}_{\Lambda^*}({ \ell},{  x})=<\psi(\cdot,{  x}),\phi_{\Lambda^*}^{({ \ell})}>_{L^2(\Wast)}$. Then the squared norm of any $\psi\in H_0^r(\Wast;X_{ \alpha})$ equals to
\begin{equation*}
\|\psi\|^2_{H_0^r(\Wast;X_{ \alpha})}=\sum_{{ \ell}\in\Z}(1+|{ \ell}|^2)^r\left\|\hat{\psi}_{\Lambda^*}({ \ell},\cdot)\right\|^2_{X_{ \alpha}}.
\end{equation*}
\end{remark}

\bibliographystyle{alpha}
\bibliography{ip-biblio} % ../../ip-biblio/ip-biblio.bib,

\providecommand{\noopsort}[1]{}
\begin{thebibliography}{CWMT07}

\bibitem[ACWD06]{Arens2006a}
T.~Arens, S.~N. Chandler-Wilde, and J.~A. DeSanto.
\newblock On integral equation and least squares methods for scattering by
  diffraction gratings.
\newblock {\em Communications in Computational Physics}, 1:1010--1042, 2006.

\bibitem[Bha00]{Bhatt2000}
A.~K. Bhattacharyya.
\newblock Analysis of multilayer infinite periodic array structures with
  different periodicities and axes orientations.
\newblock {\em IEEE T Antenn Propag}, 48(3), 2000.

\bibitem[BS94]{Brenn1994}
S.~C. Brenner and L.~R. Scott.
\newblock {\em The Mathematical Theory of Finite Element Methods}.
\newblock Springer, New York, 1994.

\bibitem[CE10]{Chand2010}
S.~N. {Chandler-Wilde} and J.~Elschner.
\newblock Variational approach in weighted {S}obolev spaces to scattering by
  unbounded rough surfaces.
\newblock {\em SIAM. J. Math. Anal.}, 42:2554--2580, 2010.

\bibitem[CM05]{Chand2005}
S.~N. {Chandler-Wilde} and P.~Monk.
\newblock Existence, uniqueness, and variational methods for scattering by
  unbounded rough surfaces.
\newblock {\em SIAM. J. Math. Anal.}, 37:598--618, 2005.

\bibitem[Coa12]{Coatl2012}
J.~Coatl{\'e}ven.
\newblock {Helmholtz equation in periodic media with a line defect}.
\newblock {\em {J. Comp. Phys.}}, 231:1675--1704, 2012.

\bibitem[CWMT07]{Chand2007}
S.~N. Chandler-Wilde, P.~Monk, and Martin Thomas.
\newblock The mathematics of scattering by unbounded, rough, inhomogeneous
  layers.
\newblock {\em Journal of Computational and Applied Mathematics}, 204:549--559,
  2007.

\bibitem[CWR96]{Chand1996}
S.~N. Chandler-Wilde and C.R. Ross.
\newblock Scattering by rough surfaces: the {D}irichlet problem for the
  {H}elmholtz equation in a non-locally perturbed half-plane.
\newblock {\em Math. Meth. Appl. Sci.}, 19:959--976, 1996.

\bibitem[CWRR02]{Chand2002}
S.~N. Chandler-Wilde, M.~Rahman, and C.~R. Ross.
\newblock A fast two-grid and finite section method for a class of integral
  equations on the real line with application to an acoustic scattering problem
  in the half-plane.
\newblock {\em Numer. Math.}, 93:1--51, 2002.

\bibitem[CWRZ99]{Chand1999}
S.N. Chandler-Wilde, C.R. Ross, and B.~Zhang.
\newblock Scattering by infinite one-dimensional rough surfaces.
\newblock {\em Proceedings of the Royal Society {A}}, 455:3767--3787, 1999.

\bibitem[CWZ98a]{Chand1998a}
S.~N. Chandler-Wilde and B.~Zhang.
\newblock Electromagnetic scattering by an inhomogeneous conducting or
  dielectric layer on a perfectly conducting plate.
\newblock {\em Proc. R. Soc. Lond. A}, 454:519--542, 1998.

\bibitem[CWZ98b]{Chand1998}
S.~N. Chandler-Wilde and B.~Zhang.
\newblock A uniqueness result for scattering by infinite dimensional rough
  surfaces.
\newblock {\em SIAM J. Appl. Math.}, 58:1774--1790, 1998.

\bibitem[CWZ99]{Chand1999a}
S.~N. Chandler-Wilde and B.~Zhang.
\newblock Scattering of electromagnetic waves by rough interfaces and
  inhomogeneous layers.
\newblock {\em SIAM J. Math. Anal.}, 30:559--583, 1999.

\bibitem[FJ15]{Fliss2015}
S.~Fliss and P.~Joly.
\newblock {Solutions of the time-harmonic wave equation in periodic waveguides:
  asymptotic behaviour and radiation condition}.
\newblock {\em {Arch. Rational Mech. Anal.}}, 2015.

\bibitem[HL11]{Hadda2011}
H.~Haddar and A.~Lechleiter.
\newblock Electromagnetic wave scattering from rough penetrable layers.
\newblock {\em SIAM J. Math. Anal.}, pages 2418--2443, 2011.

\bibitem[HLP88]{Hardy1988}
G.~H. Hardy, J.~E. Littlewood, and G.~P{\'{o}}lya.
\newblock {\em Inequalities}.
\newblock Cambridge Mathematical Library. Cambridge University Press, 2nd
  edition, 1988.

\bibitem[HLQZ15]{Hu2015}
G.~Hu, X.~Liu, F.~Qu, and B.~Zhang.
\newblock Variational approach to scattering by unbounded rough surfaces with
  {N}eumann and generalized impedance boundary conditions.
\newblock {\em Commun. Math. Sci.}, 13(2):511--537, 2015.

\bibitem[HN15]{Hadda2015}
H.~Haddar and T.~P. Nguyen.
\newblock {Volume integral method for solving scattering problems from locally
  perturbed periodic layers}.
\newblock In {\em WAVES 2015 Proceed.}, KIT, Karlsruhe, 2015.

\bibitem[HN17]{Hadda2017}
H.~Haddar and T.~P. Nguyen.
\newblock {Sampling methods for reconstructing the geometry of a local
  perturbation in unknown periodic layers}.
\newblock {\em Comput. Math. Appl.}, 74(11):2831--2855, 2017.

\bibitem[Lec17]{Lechl2016}
A.~Lechleiter.
\newblock The {F}loquet-{B}loch transform and scattering from locally perturbed
  periodic surfaces.
\newblock {\em J. Math. Anal. Appl.}, 446(1):605--627, 2017.

\bibitem[Li12]{Li2012}
P.~Li.
\newblock Analysis of the scattering by an unbounded rough surface.
\newblock {\em Math. Meth. Appl. Sci.}, 35:2166--2184, 2012.

\bibitem[LN15]{Lechl2015e}
A.~Lechleiter and D.-L. Nguyen.
\newblock {Scattering of {H}erglotz waves from periodic structures and mapping
  properties of the {B}loch transform}.
\newblock {\em {Proc. Roy. Soc. Edinburgh Sect. A}}, 231:1283--1311, 2015.

\bibitem[LR10]{Lechl2009}
A.~Lechleiter and S.~Ritterbusch.
\newblock {A variational method for wave scattering from penetrable rough
  layers}.
\newblock {\em IMA J. Appl. Math.}, 75:366--391, 2010.

\bibitem[LZ17a]{Lechl2016a}
A.~Lechleiter and R.~Zhang.
\newblock A convergent numerical scheme for scattering of aperiodic waves from
  periodic surfaces based on the {F}loquet-{B}loch transform.
\newblock {\em SIAM J. Numer. Anal}, 55(2):713--736, 2017.

\bibitem[LZ17b]{Lechl2017}
A.~Lechleiter and R.~Zhang.
\newblock A {F}loquet-{B}loch transform based numerical method for scattering
  from locally perturbed periodic surfaces.
\newblock {\em SIAM J. Sci. Comput.}, 39(5):B819--B839, 2017.

\bibitem[LZ17c]{Lechl2016b}
A.~Lechleiter and R.~Zhang.
\newblock Non-periodic acoustic and electromagnetic scattering from periodic
  structures in 3d.
\newblock {\em Comput. Math. Appl.}, 74(11):2723--2738, 2017.

\bibitem[MACK00]{Meier2000}
A.~Meier, T.~Arens, S.~N. {Chandler-Wilde}, and A.~Kirsch.
\newblock A {N}ystr{\"o}m method for a class of integral equations on the real
  line with applications to scattering by diffraction gratings and rough
  surfaces.
\newblock {\em J. Int. Equ. Appl.}, 12:281--321, 2000.

\bibitem[Str98]{Stryc1998}
B.~Strycharz.
\newblock An acoustic scattering problem for periodic, inhomogeneous media.
\newblock {\em Math. Method Appl. Sci.}, 21(10):969--983, 1998.

\bibitem[ZCW98]{Zhang1998}
B.~Zhang and S.~N. Chandler-Wilde.
\newblock Acoustic scattering by an inhomogeneous layer on a rigid plate.
\newblock {\em SIAM J. Appl. Math.}, 58(6):1931--1950, 1998.

\bibitem[ZCW03]{Zhang2003}
B.~Zhang and S.~N. Chandler-Wilde.
\newblock Integral equation methods for scattering by infinite rough surfaces.
\newblock {\em Math. Meth. Appl. Sci.}, 26:463--488, 2003.

\bibitem[Zha18]{Zhang2017e}
R.~Zhang.
\newblock A high order numerical method for scattering from locally perturbed
  periodic surfaces.
\newblock {\em accepted by SIAM J. Sci. Comput.}, 2018.

\end{thebibliography}

\end{document}